\documentclass[a4paper] {article}
[12pt]
    
\usepackage[top=2.5cm, bottom=2.5cm, left=1.8cm, right=1.8cm]{geometry}

\makeatletter
\renewcommand*\l@section{\@dottedtocline{1}{1.5em}{2.3em}}
\makeatother

\usepackage{amsfonts}
\usepackage{amssymb}
\usepackage[T1]{fontenc}

\usepackage{tikz}
\usetikzlibrary{calc}

\usepackage{CJK}
\usepackage{amsmath}
 
\usepackage{amsfonts}
\usepackage{amssymb}
\usepackage{amsthm}
\usepackage{amssymb}
\usepackage{enumerate}
\usepackage[calc]{picture}
\usepackage[all,cmtip]{xy}

\usepackage[mathscr]{eucal}
\usepackage{eqlist}

\usepackage{color}
\usepackage{abstract} 
\usepackage[T1]{fontenc}
 
\theoremstyle{plain}
\newtheorem{theorem}{Theorem}
\newtheorem{proposition}[theorem]{Proposition}
\newtheorem{lemma}[theorem]{Lemma}

\newtheorem{example}[theorem]{Example}
\newtheorem{corollary}[theorem]{Corollary}

\theoremstyle{definition}
\newtheorem{definition}{Definition}

\usepackage{etoolbox}
\newtheoremstyle{myrem}%name
 {3pt}%Space above
 {3pt}%Space below
 {\normalsize}%Body font
 { }%Indent amount
 {\itshape}% Theorem head font
 {:}%Punctuation after theorem head
 { }%Space after theorem head 2
 {}%Theorem head spec (can be left empty, meaning ?normal?)

 \theoremstyle{myrem}
 \newtheorem{remark}{Remark}
 \appto\remark{\leftskip\parindent}
 \appto\remark{\rightskip\parindent}

\numberwithin{equation}{section}
\numberwithin{theorem}{section}

\begin{document}

\begin{center}
{\Large {\textbf {
Differential  calculus  for    free  algebra  and  constrained  homology   
}}}
 \vspace{0.58cm}\\

Shiquan Ren%,   Dawei  Shen 

%The authors are supported by  The Office of China Postdoctoral Council, China Postdoctoral Science Foundation (no. 2020M680494),  Natural Science Foundation of China (no. 12001310), Scientific Research Project of Hebei Education Department (no. 
%QN2019333), Natural Fund Project of Cangzhou Science and Technology Bureau (no. 197000002) and Project
%of Cangzhou Normal University (no. xnjjl1902).

\bigskip

\begin{quote}
\begin{abstract}

In  {[Discrete  differential  calculus  on  simplicial  complexes and  constrained homology,  
  \emph{Chin.  Ann.  Math.  Ser.  B}  {\bf44}(4),  615-640,  2023]},  the  constrained  (co)homology  for  simplicial  complexes   and  independence hypergraphs  is  constructed  via  differential  calculus  on  discrete  sets. 
  In  this  paper,  we   study  the  differential  calculus  for  free  algebra  and  subsequently 
  study the  homomorphisms  of  constrained  (co)homology  induced  by  inclusions  of  simplicial  complexes  and  independence hypergraphs.  
  We  apply  the  differential  calculus  to hypergraphs.  
  We  realize  simplicial  complexes  and  independence  
  hypergraphs  as  certain  invariant  traces  of  hypergraphs. 
  As  an  application,  
  we  give  the  constrained persistent  (co)homology  for  filtrations  of  simplicial  complexes  and 
  filtrations  of  independence  hypergraphs.  
  \end{abstract}

\bigskip

{ {\bf 2020 Mathematics Subject Classification.}  	Primary  55U10,  	55U15,  Secondary  	53A45,  08A50  
}

{{\bf Keywords and Phrases.}   hypergraphs,   simplicial  complexes,   discrete calculus,  homology  }

\end{quote}

\end{center}

\bigskip

\section{Introduction}

Since  1990's,  the  topic   to    extend  the  classical  notion  of  
simplicial  (co)homology  to  various  combinatorial  settings  have been  extensively  studied.  
 For  example, 
  the  weighted  homology  of  weighted  simplicial  complexes  by   Dawson, R. J. MacG.  \cite{daw},  
  the  path  complex  and  path  homology for  digraphs  by  A.  Grigor'yan,  Y.  Lin, Y.  Muranov  and  S.-T. Yau  \cite{lin2,lin3,lin6},   
 etc.

In  1994,   A. Dimakis  and    F. M\"{u}ller-Hoissen  \cite{d1,d2}    initially  studied  the  differential  calculus  on  discrete  sets.     Since  2010's,  A.  Grigor'yan,  Y.  Lin, Y.  Muranov   and   S.-T. Yau  \cite{lin2,lin3,lin6} further  investigated  the  differential  calculus  on  discrete  sets  and  constructed  the  
path  homology  theory  for  digraphs.  
   In  2023,  the  present author  \cite{camb2023}   investigated  the     differential  calculus  on  vertices   and   constructed  the   constrained   homology  for  simplicial  complexes 
   as  well as  the  constrained  cohomology  for  independence  hypergraphs.      
    Later,  the  results  in   \cite{camb2023}  are   applied  to     some  K\"unneth-type  formulae  for  random  hypergraphs  and  random  simplicial  complexes  by   C.  Wu,  J.  Wu  and 
     the  present  author   \cite{jktr2023}.

     In  this   paper,   we   investigate   the  homomorphisms  of  constrained  homology  induced  by  inclusions  of  augmented  simplicial  complexes  as  well  as  the  homomorphisms  of  constrained  cohomology  induced  by  inclusions  of   augmented  independence hypergraphs     by  
    studying   the   differential  calculus  for  the  free  algebra.    
    By  applying the  differential  calculus,  we  prove  that  certain  invariant  traces  
    of  a   hypergraph  are    simplicial  complexes  and  independence  hypergraphs.   
    This  paper  is        a  subsequent  work  of   of   \cite{camb2023}.

 Firstly,  as  
 the  algebraic  foundation   of  this  paper,  
 we  study  the  differential  calculus  for  the  free  algebra 
 and  construct  some chain  complexes  and  chain  maps    in  Sections~\ref{s2}-\ref{s4}.

    Let  $S$  be  a   set  whose  elements are  vertices. 
    Let  $[2^S]$  be  the  collection  of  all  the   finite  subsets  of  $S$  including  the  epmtyset  
    $\emptyset$.   
    Let  $R$  be  a  commutative  ring   with  unit.  
    Let  $ R\langle  S\rangle_n $  be  the    free  $R$-module  generated  
    by  the  words  on  $S$  of  length  $n+1$  where  the  emptyset   $\emptyset$  is  
    regarded  as  a  word  of  length  $0$.   
    For  any  $s\in  S$,  consider  the  $R$-linear  map   
    %\begin{eqnarray*}
   $ \frac{\partial }{\partial  s}:     R\langle  S\rangle_n\longrightarrow    R\langle  S\rangle_{n-1} $   %\end{eqnarray*}
    given  by      \footnote[1]{For  any  $s,t\in  S$,  we  use  the  notation  $\delta(s,t)=1$  if  $s=t $  and  $\delta(s,t)= 0$  otherwise.  }
    \begin{eqnarray*}
     \frac{\partial }{\partial  s}(s_0s_1\ldots  s_n) = \sum_{i=0}^n  (-1)^i\delta(s,s_i)  s_0\ldots\widehat{s_i} \ldots  s_n
    \end{eqnarray*}
    and  the  $R$-linear  map
      %\begin{eqnarray*}
    $ds:    R\langle  S\rangle_n\longrightarrow    R\langle  S\rangle_{n+1} $ 
    %\end{eqnarray*}
    given by  
    \begin{eqnarray*}
    ds  (s_0s_1\ldots  s_n)=\sum_{i=0}^{n+1}  (-1)^i s_0\ldots s_{i-1} s s_i\ldots   s_n. 
    \end{eqnarray*}
   We  prove  the  simplicial  identities  in   Proposition~\ref{le-0a.1}
    and      the  Leibniz  rules  in   Proposition~\ref{pr-mzmza}    
    for  $\partial/\partial  s$  and  $ds$.

        Consider  the  exterior  algebras   $\wedge  ( {\partial}/{\partial  s}\mid  s\in  S  )$
      and  $\wedge  (d  s \mid  s\in  S  )$  generated  by   $ {\partial}/{\partial  s}$  
      for  all  $s\in  S$  and  by  $ds$  for  all  $s\in  S$  respectively.   
      For  $\alpha\in   \wedge^{2k+1} ( {\partial}/{\partial  s}\mid  s\in  S  )$  and    $\omega\in  \wedge^{2k+1}  (  d s \mid  s\in  S  )$,  
      we  have  chain  complexes   
      \begin{eqnarray*}
 (\tilde R\langle  S\rangle_*, \alpha,q): &&  \cdots    \overset{\alpha}{\longrightarrow}   R\langle   S\rangle_{p( 2k+1)+q}\overset{\alpha }{\longrightarrow}  R\langle   S\rangle_{(p-1)(2k+1)+q}\overset{\alpha }{\longrightarrow}   \cdots,\\
 (\tilde R\langle  S\rangle_*, \omega,q): &&  \cdots    \overset{\omega }{\longrightarrow}  R\langle   S\rangle_{p ( 2k+1)+q}\overset{\omega }{\longrightarrow}  R\langle   S\rangle_{(p+1)(2k+1)+q}\overset{\omega }{\longrightarrow}   \cdots.  
   \end{eqnarray*}
        For     $\beta\in   \wedge^{2l}  ({\partial}/{\partial  s}\mid  s\in  S )$  and    $\mu\in  \wedge^{2l}  (  d s \mid  s\in  S  )$,   
        we  have  chain  maps    $
 \beta:   (\tilde R\langle  S\rangle_*, \alpha,q)\longrightarrow    (\tilde R\langle  S\rangle_*, \alpha,q-2l) $ 
  and  
  $ \mu:   (\tilde R\langle  S\rangle_*, \omega,q)\longrightarrow    (\tilde R\langle  S\rangle_*, \omega,q+2l)$.  
  In  Theorem~\ref{pr-bmaq17},  
  we  prove  the  functoriality  of  the  chain  maps  $\beta$  and  $\mu$.  
In  particular, for  $R$  the  complex  numbers  $\mathbb{C}$,   we  give  some  duality  for  $(\tilde {\mathbb{C}}\langle  S\rangle_*, \alpha,q)$  and  $(\tilde {\mathbb{C}}\langle  S\rangle_*, \omega,q)$  in  Proposition~\ref{pr-mvxza1}  as  a  by-product  of  \cite[ Eq.  (3.5),   Lemma~3.2  and  Proposition~3.1]{camb2023}.

Suppose  $S$  has  a  total  order.  We  call a  word  $s_0s_1\ldots  s_n$  on  $S$  
{\it  simplicial  acyclic}  if   $s_0,s_1,\ldots,  s_n$  are  distinct  and  
$s_0\prec  s_1\prec\cdots  \prec  s_n$. 
  Consider  the     free  $R$-module  
$\mathcal{F}\langle  S\rangle_n$   
spanned  by  the  simplicial  acyclic  words  of  length  $n+1$  on  $  S$.  
In  Section~\ref{s4},  
we  construct  
   \begin{eqnarray*}
  (\tilde {\mathcal{F}}\langle  S\rangle_*, \alpha,q):  &&  \cdots     \overset{\alpha}{\longrightarrow}  {\mathcal{F}}\langle   S\rangle_{p ( 2k+1)+q}\overset{\alpha }{\longrightarrow}  {\mathcal{F}}\langle   S\rangle_{(p-1)(2k+1)+q}\overset{\alpha }{\longrightarrow}   \cdots 
  \end{eqnarray*}
  as   a   sub-chain  complex  of  $(\tilde R\langle  S\rangle_*, \alpha,q)$  for  
 $\alpha\in   \wedge^{2k+1} ( {\partial}/{\partial  s}\mid  s\in  S  )$  
 and  construct  
   \begin{eqnarray*}
(\tilde {\mathcal{F}}\langle  S\rangle_*, \omega,q): &&  \cdots    \overset{\omega }{\longrightarrow}  {\mathcal{F}}\langle   S\rangle_{p ( 2k+1)+q}\overset{\omega }{\longrightarrow}  {\mathcal{F}}\langle   S\rangle_{(p+1)(2k+1)+q}\overset{\omega_{p+1}}{\longrightarrow}   \cdots
   \end{eqnarray*}
 as  a  quotient  chain  complex  of   $(\tilde R\langle  S\rangle_*, \omega,q)$  
 for       $\omega\in  \wedge^{2k+1}  (  d s \mid  s\in  S  )$.  
     For   $\beta\in   \wedge^{2l}  ({\partial}/{\partial  s}\mid  s\in  S )$  and    $\mu\in  \wedge^{2l}  (  d s \mid  s\in  S  )$,   
        we  have  chain  maps    $
 \beta:   (\tilde {\mathcal{F}}\langle  S\rangle_*, \alpha,q)\longrightarrow    (\tilde {\mathcal{F}}\langle  S\rangle_*, \alpha,q-2l) $ 
  and  
  $ \mu:   (\tilde {\mathcal{F}}\langle  S\rangle_*, \omega,q)\longrightarrow    (\tilde {\mathcal{F}}\langle  S\rangle_*, \omega,q+2l)$.  
 We  give  the  functoriality  of  these  chain  maps  $\beta$  and  $\mu$  
 in  Proposition~\ref{pr-bmaq17-co}  as  an  analogue  of  Theorem~\ref{pr-bmaq17}.   
 We  study  the  functoriality  of  the  induced  maps  of $\wedge  (\frac{\partial}{\partial  s}\mid  s\in  S  )$
      and  $\wedge  (d  s \mid  s\in  S  )$  on  $\mathcal{F}\langle  S\rangle_n$   
 in  Proposition~\ref{pr-008.990}.

       Secondly,   as  the  main  part   of  this  paper,  we  apply  the  differential  calculus  for  the  free  algebra  to  
       study  the  functoriality  of  the  constrained  (co)homology  of  simplicial  complexes  and  
       independence  hypergraphs
        in  Sections~\ref{s---5}-\ref{s---7}.

     We  define  an  {\it augmented  hypergraph}  $\mathcal{H}$  as  a  subset  of  $[2^S]$
      and  call  an  element    in  $\mathcal{H}$  consisting  of  $n+1$  vertices  an   {\it   $n$-hyperedge}  for  $n\geq  -1$. 
      We  call  $\mathcal{H}$  an  {\it  augmented  simplicial  complex}  and denote  it as  
      $\mathcal{K}$  if  for  any  $\sigma\in \mathcal{H}$  and  any  nonempty  subset  $\tau$  
      of  $\sigma$,  it  holds  $\tau\in  \mathcal{H}$.  
      We  call  $\mathcal{H}$  an  {\it  augmented  independence  hypergraph}  and denote  it as  
      $\mathcal{L}$  if  for  any  $\sigma\in \mathcal{H}$  and  any     superset  $\tau$  
      of  $\sigma$,  it  holds  $\tau\in  \mathcal{H}$.  
      For  any  augmented  hypergraph  $\mathcal{H}$  on  $S$,  
      we  have  its  associated   augmented   simplicial  complex  $\Delta\mathcal{H}$
      as  the  smallest  augmented  simplicial  complex  on  $S$  containing  $\mathcal{H}$,  
      its lower-associated  augmgented  simplicial complex  $\delta\mathcal{H}$
      as  the  largest  augmented  simplicial  complex  on  $S$  contained  in  $\mathcal{H}$,  
      its  associated  augmented  independence  hypergraph  $\bar\Delta\mathcal{H}$
      as  the  smallest  augmented  independence  hypergraph  on  $S$ containing  $\mathcal{H}$,  
      its  lower-associated  augmented independence  hypergraph  $\bar\delta\mathcal{H}$
      as  the  largest  augmented  independence  hypergraph  on  $S$ contained  in  $\mathcal{H}$,  
      its   global   complement   $\gamma_S\mathcal{H}$ as  $[2^S]\setminus \mathcal{H}$  
       and  its   local  complement  $\Gamma_S\mathcal{H}$  consisting  of  all the  
       hyperedges  $S\setminus\sigma$  for  $\sigma\in \mathcal{H}$.        
      The     map  algebra   is  the  semigroup   generated  by  all  the  compositions  of  
      the  operations  $\Delta$,  $\delta$,  
      $\bar\Delta$,  $\bar\delta$,  $\gamma_S$,  $\Gamma_S$  
      equipped  with  the  binary  operations  the  intersection  $\cap $,  the  union  $\cup$,  
      and  the join  $*$.
      Let  $T$  be  a  subset  of  $S$.  
      The    trace   of  $\mathcal{H}$  on  $T$  
      is  the  augmented  hypergraph  consisting  of  the  hyperedges  
      $\sigma\cap  T$  where  $\sigma\in \mathcal{H}$.  
      In  Theorem~\ref{thh=5.1},  we  give  some  partial  relations  between  the  trace operation  
      and   the   map  algebra.

         Let  $\mathcal{H}$  be  an  augmented  hypergraph  on  $S$.   
         Let  $R_n(\mathcal{H}) $   be  the  
     free  $R$-module  spanned  by  the  $n$-hyperedges  in  $\mathcal{H}$.  
      Then  
       $R_n(\mathcal{H}) $ 
       is  a  sub-$R$-module  of  
       $\mathcal{F}\langle  S\rangle  _n$.  
        Let   $    R_*(\mathcal{H})= \oplus_{n\geq   0}  R_n(\mathcal{H})$  and  
          $ \tilde   R_*(\mathcal{H})= \oplus_{n\geq   -1}  R_n(\mathcal{H})$. 
          Then  
              $ \tilde   R_*(\mathcal{H})$ 
              is  a  graded  sub-$R$-module  of   $\tilde {\mathcal{F}}\langle  S\rangle_*$.                
         For  any   $s\in  S$,    
         we  say  that 
          $      R_*(\mathcal{H})$  is      $\partial/\partial  s$-invariant    if  
 $\partial_i/\partial  s:       R_{n+1}(\mathcal{H})\longrightarrow    R_{n}(\mathcal{H})$  is  well-defined  for   any      $n  \in \mathbb{N} $  and  any  
$0\leq   i\leq   n+1$ 
and     say  that    $      R_*(\mathcal{H})$  is    $d  s$-invariant    if  
$d_i  s:  R_n(\mathcal{H})\longrightarrow  R_{n+1}(\mathcal{H})$  is  well-defined  for   any   $n \in \mathbb{N}$  and  any  
$0\leq   i\leq   n+1$.    
 Let  $S( \mathcal{H},\partial)$   be   the  subset  of  $S$  consisting  of  all  $s\in  S$  such  that  
    $    R_*(\mathcal{H})$  is   $\partial/\partial  s$-invariant  and   let   $S(\mathcal{H},d)$   be   the  subset  of  $S$  consisting  of  all  $s\in  S$  such  that  
    $      R_*(\mathcal{H})$  is   $d  s$-invariant. 
      In   Theorem~\ref{pr-5.5aaa},  
      we  prove  (1)  for  any  augmented  hypergraph  $\mathcal{H}$,    the  trace  
      $\mathcal{H}\mid_{S( \mathcal{H},\partial)}$
      is  an  augmented   simplicial  complex    on  $S( \mathcal{H},\partial)$;  
      (2)  for  any  hypergraph  $\mathcal{H}$,  the  trace  $\mathcal{H}\mid _{S(\mathcal{H},d)}$ 
      is an     independence  hypergraph  on  $S(\mathcal{H},d)$.

      Let  $\mathcal{K}$  be  an  augmented  simplicial  complex  on  $S$.  
      Let  $\mathcal{L}$  be  an  augmented  independence  hypergraph  on   $S$.  
      In  Section~\ref{s---7},  
       For  $\alpha\in   \wedge^{2k+1} ( {\partial}/{\partial  s}\mid  s\in  S  )$  and    $\omega\in  \wedge^{2k+1}  (  d s \mid  s\in  S  )$,
       we  construct  
      \begin{eqnarray*}
       (\tilde  R_*(\mathcal{K}), \alpha,q): ~~~   \cdots    \overset{\alpha }{\longrightarrow} R_{p ( 2k+1)+q}(\mathcal{K})  \overset{\alpha }{\longrightarrow}  R_{(p-1)(2k+1)+q}(\mathcal{K})  \overset{\alpha }{\longrightarrow}   \cdots 
      \end{eqnarray*}
      as  a  sub-chain  complex  of  
        $(\tilde{\mathcal{F}}\langle  S\rangle_*, \alpha,q)$
         and  construct  
         \begin{eqnarray*}
        (\tilde  R_*(\mathcal{L}), \omega,q): ~~~   \cdots      \overset{\omega }{\longrightarrow}R_{p( 2k+1)+q}(\mathcal{L})  \overset{\omega }{\longrightarrow} R_{(p+1)(2k+1)+q}(\mathcal{L})  \overset{\omega }{\longrightarrow}   \cdots
         \end{eqnarray*}
         as  a  sub-chain  complex  of  $(\tilde {\mathcal{F}}\langle  S\rangle_*, \omega,q)$.   
          The     constrained  homology  $H_*(\mathcal{K},\alpha,q)$  of  $\mathcal{K}$  is  
          defined  as  the  homology  group  of  $(\tilde  R_*(\mathcal{K}), \alpha,q)$
          (cf.  \cite[Definition~4.3]{camb2023}) 
          and  the  constrained  cohomology  $ H^*(\mathcal{L},\omega,q)$    of  $\mathcal{L}$  
          is  defined  as  the  homology  group  of  $ (\tilde  R_*(\mathcal{L}), \omega,q)$ (cf.  \cite[Definition~4.4]{camb2023}). 
          For   $\beta\in   \wedge^{2l}  ({\partial}/{\partial  s}\mid  s\in  S )$  and    $\mu\in  \wedge^{2l}  (  d s \mid  s\in  S  )$,   there  are   induced  homomorphisms 
          $
 \beta:   H_*(\mathcal{K},\alpha,q)\longrightarrow    H_*(\mathcal{K},\alpha, q-2l) $ 
  and  
  $ \mu:    H^*(\mathcal{L},\omega,q)\longrightarrow   H^*(\mathcal{L},\omega, q+2l)$
   (cf.  \cite[Theorem~4.2 and Theorem~4.4]{camb2023}).  
   In  Theorem~\ref{th-1daza},  
   we  prove  the  functoriality  of  these  homomorphisms   $\beta$  and  $\mu$  
   as  a  consequence  of 
  Proposition~\ref{pr-bmaq17-co}   and  Theorem~\ref{pr-bmaq17}. 
   In  Theorem~\ref{th-77.3},  
   we  prove  the  functoriality  of  the  constrained  homology  with respect to inclusions of 
   augmented simplicial  complexes
   as  well  as  the  functoriality  of the  constrained  cohomology  
   with  respect to  inclusions  of  augmented  independence  hypergraphs.  
   Some  Mayer-Vietoris  sequences  for  the  constrained  homology of  augmented 
   simplicial  complexes  and  the  constrained  cohomology  of  augmented  
   independence hypergraphs  are  discussed  by  the  end  of  Section~\ref{s---7}.

      Summarizing   Sections~\ref{s---5}-\ref{s---7},   given  an  augmented  hypergraph  
      $\mathcal{H}$  on  $S$,  we  have  augmented  simplicial  complexes  
      $\Delta\mathcal{H}$  and  $\delta\mathcal{H}$  on  $S$,  
      an  augmented  simplicial  complex 
       $\mathcal{H}\mid_{S( \mathcal{H},\partial)}$
         on  $S( \mathcal{H},\partial)$,  
      augmented  independence  hypergraphs  $\bar\Delta\mathcal{H}$  and  $\bar\delta\mathcal{H}$  
      on  $S$  
      and  
      an  augmented  independence  hypergraph 
      $\mathcal{H}\mid _{S(\mathcal{H},d)}$ 
       on  $S(\mathcal{H},d)$.  
      The  constrained  homology  $H_*(-,\alpha,q)$  applies to  
      $\Delta\mathcal{H}$,  $\delta\mathcal{H}$  and  $\mathcal{H}\mid_{S( \mathcal{H},\partial)}$
      while  
      the  constrained  cohomology  $H^*(-,\omega,q)$  applies to  $\bar\Delta\mathcal{H}$,  $\bar\delta\mathcal{H}$  and  $\mathcal{H}\mid_{S( \mathcal{H},\partial)}$.  
      By  (\ref{eq-vmbg1})  and  (\ref{eq-vmbg2}),  
      the  canonical   inclusion  $\varphi: \mathcal{H}\longrightarrow \mathcal{H}'$  of  two augmented  hypergraphs 
           $\mathcal{H}$  and  $\mathcal{H}'$  on  $S$
       induces  canonical  inclusions  of  augmented  simplicial complexes  
       $\Delta\varphi: \Delta\mathcal{H}\longrightarrow \Delta\mathcal{H}'$
       and  $\delta\varphi: \delta\mathcal{H}\longrightarrow \delta\mathcal{H}'$ 
       as  well  as     canonical  inclusions  of  augmented  independence  hypergraphs  
       $\bar\Delta\varphi:  \bar\Delta\mathcal{H}\longrightarrow \bar\Delta\mathcal{H}'$
       and  $\bar\delta\varphi: \bar\delta\mathcal{H}\longrightarrow \bar\delta\mathcal{H}'$.   
       The  functoriality  in  Theorem~\ref{th-1daza}  and   Theorem~\ref{th-77.3} 
       applies  to  these  canonical  inclusions.

      Thirdly,  as  an  application  of  this  paper,   we  consider the   constrained  persistent  homology for  filtrations  of  augmented 
      simplicial  complexes  and  the   constrained  persistent  cohomology  for  
      filtrations  of  augmented  independence  hypergraphs,  in  Section~\ref{s--8}.  
      We  give  the  persistent  version of Theorem~\ref{th-1daza}  in  Corollary~\ref{co-8.1}.  
      As  consequences  of  Section~\ref{s---7},  we  discuss  some  Mayer-Vietoris sequences  for  the   constrained  persistent  homology for  filtrations  of  augmented 
      simplicial  complexes  in  Corollary~\ref{co-8.2} 
      and   for  the   constrained  persistent  cohomology  for  
      filtrations  of  augmented  independence  hypergraphs in  Corollary~\ref{co-8.3}.

        The rest  of  this  paper  is  organized as  follows.  
    % In  Sections~\ref{s2}-\ref{s4},  we  give  some  algebraic  foundation.  
     In  Section~\ref{s2},  
  we  prove  some  simplicial  identities  and  the  Leibniz  rules  for  the  differential  calculus  
  on  the  free  algebra.  
   In  Section~\ref{s3},  applying    the  differential  calculus,  
  we  construct  some  chain  complexes  and    chain  maps  for  the  free  algebra.
     We  prove  the  functoriality.  
  In  Section~\ref{s4},  
  we  construct    chain  complexes  generated  by     simplicial  acyclic  words
  as   sub-chain  complexes  and  quotient  chain  complexes  of  the  chain  complex  for  the  free  algebra.  
    % In  Sections~\ref{s---5}-\ref{s---7},  we  study  the  functoriality  of  the  constrained  (co)homology     of  simplicial  complexes  and  
    %   independence  hypergraphs.  
       In  Section~\ref{s---5},  we  study  the  traces  of  augmented hypergraphs  and  their  relation  
       with  the  map  algebra  for  augmented  hypergraphs.   
       In  Section~\ref{s---6},  applying  the  differential  calculus,
       we prove that the  invariant traces  of  augmented   hypergraphs  are 
       augmented  simplicial  complexes  and  augmented   independence  hypergraphs.  
       In  Section~\ref{s---7},  we   investigate   the  homomorphisms  of  constrained  homology  induced  by  inclusions  of  augmented  simplicial  complexes  and  the  homomorphisms  of  constrained  cohomology  induced  by  inclusions  of   augmented  independence hypergraphs.  
       We  prove the  functoriality.  
       In  Section~\ref{s--8},   we  consider  the  persistent  version  of    Section~\ref{s---7}  and     study  the  constrained persistent   (co)homology.

\section{The  free  algebra}\label{s2}

Let  $S$  be  a  set.   
Let  $\langle  S\rangle$  be  the  free  semi-group   generated  by  $S$.  
Then  $\langle  S\rangle$  consists  of  the  elements  of  the  form  $s_0 s_1\cdots  s_n$ 
where  $s_0,s_1,\ldots,s_n\in  S$.  
The  multiplication  of  $\langle   S\rangle $  is  $(s_0 \cdots  s_n) \cdot (t_0\cdots  t_m)= s_0 \cdots  s_nt_0\cdots  t_m$.  
Let  $\langle  S\rangle_n$  be  the  subset  of  $S$  consisting  of  the  elements  $s_0 s_1\cdots  s_n$ 
where  $s_0,s_1,\ldots,s_n\in  S$.  
Then  $\langle  S\rangle = \coprod_{n=0}^\infty  \langle  S\rangle_n$  and  the  multiplication  of  $S$
 sends  the  pair  $( \langle  S\rangle_n,  \langle  S\rangle_m)$  to  $ \langle  S\rangle_{n+m}$.

Let  $R$  be  a  commutative ring  with  unit  $1$  such  that  $2$  has  a  multiplicative  inverse.   
Let  $R\langle S\rangle$  be  the  free  $R$-module  spanned  by  $\langle  S\rangle$   
   and  let  $R\langle S\rangle_n$  be  the  free  $R$-module  spanned  by  $\langle  S\rangle_n$.  
Then  
$R\langle  S\rangle  =\oplus_{n=0}^\infty   R\langle S\rangle_n$.   
 Let  $R\langle  S\rangle_{-1}= R$.  
 Consider  the  augmented  free  $R$-module 
 $\tilde R\langle  S\rangle =\oplus_{n=-1}^\infty   R\langle S\rangle_n$.   
For  any  $m,n\in\mathbb{N}$,  the  multiplication  of  $\langle  S\rangle $  extends  to  a  bilinear  map   from
    $(  R\langle  S\rangle_n,   R\langle  S\rangle_m  )$ to  $  R\langle  S\rangle_{n+m}$.   We  call  $\tilde R\langle  S\rangle$  the  {\it   augmented  free  algebra}  
generated  by  $S$.  
  Let  $s\in  S$.   
We    have    an  $R$-linear map 
\begin{eqnarray}\label{eq-2.za1}
\overrightarrow{\frac{\partial}{\partial  s}}:~~~   R\langle  S\rangle_n\longrightarrow  \prod_{n+1} R\langle  S\rangle_{n-1}
\end{eqnarray}
given  by 
\begin{eqnarray}\label{eq-apr.8}
\overrightarrow {\frac{\partial}{\partial  s}}=\Big(\frac{\partial_0}{\partial s},   \frac{\partial_1}{\partial s},\ldots, \frac{\partial_n}{\partial s}\Big),  
\end{eqnarray}
where   for each  $0\leq i\leq n$,  the $i$-th coordinate of (\ref{eq-apr.8})  is an  $R$-linear map
\begin{eqnarray*}
 \frac{\partial_i}{\partial  s}:~~ R\langle  S\rangle_n\longrightarrow    R\langle  S\rangle_{n-1}  
 \end{eqnarray*}
given by 
\begin{eqnarray}\label{eq-2.4ax}
\frac{\partial_i}{\partial  s}(s_0s_1\ldots  s_n)=
(-1)^i\delta(s,s_i)  s_0\ldots\widehat{s_i} \ldots  s_n.      
\end{eqnarray} 
On  the  other  hand,   we    have           an  $R$-linear map 
\begin{eqnarray}\label{eq-xvca3}
\overrightarrow {{d}s}:~~~  R\langle  S\rangle_n\longrightarrow\prod_{n+2}  R\langle  S\rangle_{n+1}
\end{eqnarray}
given  by 
\begin{eqnarray}\label{eq-apr.9}
\overrightarrow {{d}s}=(d_0 s,  d_1 s,\ldots, d_{n+1} s ),  
\end{eqnarray}  
where    for each  $0\leq i\leq n+1$,  the $i$-th coordinate of (\ref{eq-apr.9})  is an  $R$-linear map
\begin{eqnarray*}
d_i  s:~~ R\langle  S\rangle_n\longrightarrow   R\langle  S\rangle_{n+1}
 \end{eqnarray*}
given by 
\begin{eqnarray}\label{eq-2.2.1}
d_i s(s_0s_1\ldots s_n)=
(-1)^i s_0\ldots s_{i-1} s s_i\ldots   s_n.  
\end{eqnarray} 

\begin{proposition}[Simplicial  identities]
\label{le-0a.1}
For  any  $s,t\in  S$  and  any  possible  $i, j$,    the  following  identities  hold
 \begin{enumerate}[(1) ]
 \item
  $\dfrac{\partial_i}{\partial  s}\circ \dfrac{\partial_j}{\partial t}=-\dfrac{\partial_{j-1}}{\partial  t}\circ\dfrac{\partial_i}{\partial  s}$ for $i<j$;
  \item
$
\dfrac{\partial_i}{\partial  s}\circ d_j  t = \begin{cases}
- d_{j-1} t \circ\dfrac{\partial_i}{\partial  s},  & i<j,\\
\delta(s,t)  {\rm~ id},   & i=j, \\
- d_j  t \circ \dfrac{\partial_{i-1}}{\partial  s},  &  i>j; 
\end{cases}
$
\item
$d_i s \circ d_j t= -d_{j+1}  t \circ  d_i  s$  for $i\leq j$.    
\end{enumerate}
\end{proposition}

\begin{proof}
(1)  
Let   $0\leq j\leq n$  and  any $0\leq i\leq n-1$.   Then    
\begin{eqnarray*}
&&\frac{\partial_i}{\partial  s}\circ \frac{\partial_j}{\partial  t}(s_0s_1\ldots  s_n)\\
&=&\frac{\partial_i}{\partial s}\Big((-1)^j\delta(t,s_j)  s_0\ldots\widehat{s_j} \ldots  s_n  
\Big)\\
 &=&(-1)^j\delta(t,s_j) \frac{\partial_i}{\partial s} (s_0\ldots\widehat{s_j} \ldots  s_n)\\
 &=&  
 \begin{cases}
 (-1)^{i+j}\delta(t,s_j)\delta(s,s_i)  s_0\ldots\widehat{s_i}\ldots \widehat{s_j} \ldots s_n,  & 0\leq i\leq j-1;\\
 (-1)^{i+j}\delta(t,s_j)\delta(s,s_{i+1})  s_0\ldots\widehat{s_j}\ldots \widehat{s_{i+1}} \ldots s_n, &j\leq i\leq n-1.
  \end{cases}
 \end{eqnarray*}
Exchanging $s$ and $t$, 
 \begin{eqnarray*}
&&\frac{\partial_i}{\partial t}\circ \frac{\partial_j}{\partial s}(s_0s_1\ldots  s_n)
\\ &=&  
 \begin{cases}
 (-1)^{i+j}\delta(s,s_j)\delta(t,s_i)  s_0\ldots\widehat{s_i}\ldots \widehat{s_j} \ldots  s_n,  & 0\leq i\leq j-1;\\
 (-1)^{i+j}\delta(s,s_j)\delta(t,s_{i+1})  s_0\ldots\widehat{s_j}\ldots \widehat{s_{i+1}} \ldots  s_n, &j\leq i\leq n-1. 
  \end{cases}
 \end{eqnarray*}
Without  loss  of  generality,  suppose $i<j$.  Then   
\begin{eqnarray*}
\frac{\partial_i}{\partial  s}\circ \frac{\partial_j}{\partial  t}(s_0s_1\ldots  s_n)=-\frac{\partial_{j-1}}{\partial  t}\circ\frac{\partial_i}{\partial s}(s_0s_1\ldots  s_n). 
\end{eqnarray*}
By the $R$-linearity  of $ ({\partial_i}/{\partial s})\circ ({\partial_j}/{\partial t})$  and  $ ({\partial_{j-1}}/{\partial t})\circ ({\partial_i}/{\partial s})$,   we  obtain  (1).

(2) Let   $0\leq  i\leq  n$  and    $0\leq  j\leq  n+1$.

 {\sc Case~1}.  $i<j$.  Then 
\begin{eqnarray*}
 \frac{\partial_i}{\partial  s}\circ d_j  t(s_0 s_1\ldots s_n) 
&=& (-1)^{i+j} \delta( s,s_i) s_0\ldots \widehat{s_i} \ldots s_{j-1} t  s_j \ldots s_n\\
&=& -(-1)^{i+(j-1)} \delta( s,s_i) s_0\ldots \widehat{s_i} \ldots s_{j-1} t s_j \ldots s_n\\
&=&- d_{j-1} t \circ\dfrac{\partial_i}{\partial s}(s_0s_1\ldots  s_n).  
\end{eqnarray*}
 By the  $R$-linearity  of $({\partial_i}/{\partial  s})\circ d_j  t $  and $d_{j-1} t \circ( {\partial_i}/{\partial s})$,   
\begin{eqnarray*}
\frac{\partial_i}{\partial s}\circ d_j t=- d_{j-1} t \circ\dfrac{\partial_i}{\partial s}.  
\end{eqnarray*}

{\sc Case~2}.  $i=j$.  Then 
\begin{eqnarray*}
 \frac{\partial_i}{\partial  s}\circ d_i  t(s_0s_1\ldots  s_n) 
&=& (-1)^{i+i} \delta( s,t) s_0\ldots  s_{i-1} \widehat{t} s_i \ldots\ldots s_n\\
&=& \delta( s,t) s_0s_1\ldots s_n.  
\end{eqnarray*}
By the $R$-linearity of $({\partial_i}/{\partial s})\circ d_j  t$,  
\begin{eqnarray*}
\frac{\partial_i}{\partial s}\circ d_j t=\delta(s,t)  {\rm~ id}.  
\end{eqnarray*}

{\sc Case~3}.  $i>j$.  Then 
\begin{eqnarray*}
 \frac{\partial_i}{\partial s}\circ d_j t(s_0s_1\ldots  s_n) 
&=& (-1)^{i+j} \delta( s,s_i) s_0\ldots  s_{j-1} t s_j \ldots\widehat{s_{i-1}} \ldots s_n\\
&=& -(-1)^{(i-1)+j} \delta( s,s_i) s_0\ldots s_{j-1} t s_j \ldots \widehat{s_{i-1}} \ldots s_n\\
&=&- d_{j} t \circ\dfrac{\partial_{i-1}}{\partial s}(s_0s_1\ldots s_n).  
\end{eqnarray*}
By the  $R$-linearity  of $({\partial_i}/{\partial s})\circ d_j t$ and  $d_{j} t \circ({\partial_{i-1}}/{\partial s})$,    
\begin{eqnarray*}
\frac{\partial_i}{\partial s}\circ d_j t=- d_{j} t \circ\dfrac{\partial_{i-1}}{\partial s}.  
\end{eqnarray*}

Summarizing all the three cases,  we obtain (2).

(3)  Let   $0\leq i\leq  j\leq n$.   Then    
\begin{eqnarray*}
&&d_i s \circ d_j t (s_0s_1\ldots  s_n)\\
&=& (-1)^j d_i s(s_0\ldots  s_{j-1} t  s_j\ldots  s_n)\\
&=&\begin{cases}
(-1)^{i+j}s_0\ldots s_{i-1} s s_i\ldots  s_{j-1} t s_j\ldots s_n,  &i<j;\\
 s_0\ldots s_{i-1} s   t s_i\ldots s_n,&i=j;\\
  -s_0\ldots s_{i-1} t s s_i\ldots s_n,&i=j+1;\\
(-1)^{i+j}s_0\ldots s_{j-1} t s_j\ldots  s_{i-2} s s_{i-1}\ldots s_n, & i>j+1.
\end{cases}
\end{eqnarray*}
Since $i\leq j$, 
 \begin{eqnarray*}
d_i s \circ d_j t (s_0s_1\ldots  s_n)=-d_{j+1}  t \circ  d_i s (s_0s_1\ldots  s_n).   
 \end{eqnarray*}
 By the $R$-linearity  of $d_i s \circ d_j t$  and $d_{j+1}  t \circ  d_i  s$,  we obtain  (3).  
\end{proof}

Let  ${\partial}/{\partial  s}  = \sum_{i=0}^n   {\partial_i}/{\partial  s}$  and    $ds  = \sum_{i=0}^{n+1}  d_i s$.    

\begin{lemma}
\label{le-2.5.x}
For any $s,t\in  S$,    
\begin{eqnarray*}
\label{eq-2.5.7}
\frac{\partial}{\partial  s}\circ\frac{\partial}{\partial t} &=&-\frac{\partial}{\partial t}\circ \frac{\partial}{\partial s}, \\
 d s \circ  dt&=&-dt\circ ds.  
\end{eqnarray*}
\end{lemma}
\begin{proof}
 The  proof  is  an  analogue  of  \cite[Lemma~3.1  and  Lemma~3.3]{camb2023}.  
\end{proof}

Consider   the  exterior  algebras   over  $R$   
\begin{eqnarray*}
\wedge \Big(\frac{\partial}{\partial  s}\mid  s\in  S \Big)&=&  \bigoplus_{n\geq  0}   \wedge^n \Big(\frac{\partial}{\partial  s}\mid  s\in  S \Big),\\     
 \wedge ( ds \mid  s\in  S ) &=& \bigoplus_{n\geq  0}   \wedge^n ( ds \mid  s\in  S ).     
 \end{eqnarray*}

\begin{proposition}[Leibniz rules]\label{pr-mzmza}
Let  $\alpha \in   \wedge^1  ( {\partial}/{\partial  s}\mid  s\in  S  )$  and  
$\omega\in   \wedge^1 ( ds \mid  s\in  S ) $.   
Then  for  any  $\xi\in    R\langle  S\rangle_n$  and  any  $\eta\in    R\langle  S\rangle_m$,  
\begin{eqnarray*}
\alpha (\xi  \eta)&=&  \alpha(\xi)  \eta +  (-1)^{n+1}\xi  \alpha(\eta),\\
 \omega (\xi \eta)&=&  \omega(\xi) \eta + (-1)^{n+1} \xi  \omega(\eta). 
\end{eqnarray*}
\end{proposition}

\begin{proof}
By  the  $R$-linearity,  it  suffices  to  verify  the  identities  for  $\alpha= {\partial}/{\partial  s}$,   $\omega= ds$,  
$\xi = s_0 \ldots  s_n$    and  $\eta= t_0\ldots  t_m $.   In  fact,  
\begin{eqnarray*}
  \frac{\partial}{\partial  s}( s_0 \ldots  s_n  t_0\ldots  t_m) 
&=&  \sum_{i=0}^{n+m+1 } \frac{\partial_i}{\partial  s}  ( s_0 \ldots  s_n  t_0\ldots  t_m)\\
&=&\sum_{i=0}^{n}   (-1)^i  \delta(s,s_i) s_0\ldots \widehat{s_i}  \ldots s_n   t_0\ldots  t_m\\ 
&& + \sum_{i=0}^{m}   (-1)^{n+i+1}  \delta(s,t_i) s_0\ldots s_n   t_0\ldots\widehat{t_i}\ldots  t_m \\
&=& \frac{\partial}{\partial  s}  ( s_0 \ldots  s_n)   t_0\ldots  t_m +  (-1)^{n+1}  s_0 \ldots  s_n \frac{\partial}{\partial  s}( t_0\ldots  t_m ) 
\end{eqnarray*}  
and 
\begin{eqnarray*}
  ds ( s_0 \ldots  s_n  t_0\ldots  t_m) 
&=&\sum_{i=0}^{n+m+2}  d_is  ( s_0 \ldots  s_n  t_0\ldots  t_m) \\
&=&\sum_{i=0}^{n+1}     (-1)^i s_0\ldots s_{i-1} s s_i\ldots   s_n    t_0\ldots  t_m \\
&& + \sum_{i=0}^{m+1}   (-1)^{n+i+1}   s_0\ldots s_n   t_0\ldots  t_{i-1}  s {t_i}\ldots  t_m \\
&=& ds   ( s_0 \ldots  s_n)   t_0\ldots  t_m +  (-1)^{n+1}  s_0 \ldots  s_n  ds  ( t_0\ldots  t_m ).   
\end{eqnarray*}  
The  identities  are  obtained.  
\end{proof}

\section{Chain  complexes     for  the  free  algebra   }\label{s3}

 Given  two  chain  complexes  $(X,\partial_X)$ and  $(Y,\partial_Y)$   where  $X, Y$ are  graded  $R$-modules  and  $\partial_X,  \partial_Y$  are  the  boundary  maps,  we  use  $C((X,\partial_X),(Y,\partial_Y))$  to  denote the  collection  of  
 all the   chain  maps  from  $(X,\partial_X)$  to  $(Y,\partial_Y)$.  
 It  is  direct  that  $R$-linear  combinations  of  finitely  many  chain  maps  
   from  $(X,\partial_X)$  to  $(Y,\partial_Y)$  is  still  a  chain  map    from  $(X,\partial_X)$  to  $(Y,\partial_Y)$.   
   Thus  
 %Let  $R(X)$  be  the  ring  of  all  the homomorphisms  of $R$-modules    from  $X$ to  $R$.  
 %\begin{lemma}\label{le-mmm-1}
  $C((X,\partial_X),(Y,\partial_Y))$   is  an  $R$-module.  
 %\end{lemma}
 
 %\begin{proof}
% Let  $f, f':   (X,\partial_X)\longrightarrow  (Y,\partial_Y)$     be    chain  maps.  
% Then  $f\partial_X = \partial_Y  f$  and   $f'\partial_X = \partial_Y  f'$.    
%For  any  $r, r'\in  R$,  
%$rf+r'f': X\longrightarrow   Y$  is  a  graded  $R$-linear  map  such  that  
%%\begin{eqnarray*}
%(rf+r'f'  )\partial_X =  r  f \partial_X  +r' f' \partial_X=  r \partial_Y  f +   r' \partial_Y  f' 
 %=     \partial_Y  (rf+r'f'  ).  
%\end{eqnarray*}
%Thus  $rf+r'f': (X,\partial_X)\longrightarrow  (Y,\partial_Y)$  is  a  chain  map.   
%Therefore,  $C((X,\partial_X),(Y,\partial_Y))$   is  an  $R$-module.
% \end{proof}
 
 Let  $k\in  \mathbb{N}$.  Let  $\alpha\in   \wedge^{2k+1} ( {\partial}/{\partial  s}\mid  s\in  S  )$  and    $\omega\in  \wedge^{2k+1}  (  d s \mid  s\in  S  )$.  
   Since $2$  has  an multiplicative  inverse in  $R$,       $ \alpha\wedge\alpha=\omega\wedge\omega=0$.   Let  $q\in \mathbb{N}$.  We  obtain  chain  complexes 
   \begin{eqnarray*}
  &&  \cdots   \overset{\alpha_{p+2}}{\longrightarrow}    R\langle   S\rangle_{(p+1)(2k+1)+q}  \overset{\alpha_{p+1}}{\longrightarrow}   R\langle   S\rangle_{p( 2k+1)+q}\overset{\alpha_p}{\longrightarrow}  R\langle   S\rangle_{(p-1)(2k+1)+q}\overset{\alpha_{p-1}}{\longrightarrow}   \cdots,\\
 &&  \cdots   \overset{\omega_{p-2}}{\longrightarrow}    R\langle   S\rangle_{(p-1)(2k+1)+q}  \overset{\omega_{p-1}}{\longrightarrow}  R\langle   S\rangle_{p ( 2k+1)+q}\overset{\omega_p}{\longrightarrow}  R\langle   S\rangle_{(p+1)(2k+1)+q}\overset{\omega_{p+1}}{\longrightarrow}   \cdots
   \end{eqnarray*}
 denoted  as   $(\tilde R\langle  S\rangle_*, \alpha,q)$  and    $(\tilde R\langle  S\rangle_*, \omega,q)$    respectively.  
  Let   $l\in\mathbb{N}$.     Let  $\beta\in   \wedge^{2l}  ({\partial}/{\partial  s}\mid  s\in  S )$  and    $\mu\in  \wedge^{2l}  (  d s \mid  s\in  S  )$.   
 Since  $\beta\wedge\alpha=\alpha\wedge\beta$  and  $\mu\wedge\omega=\omega\wedge \mu$,  
  we  obtain   chain  maps 
  \begin{eqnarray} \label{eq-cm1}
 \beta:  &&(\tilde R\langle  S\rangle_*, \alpha,q)\longrightarrow    (\tilde R\langle  S\rangle_*, \alpha,q-2l),\\
   \mu:  &&(\tilde R\langle  S\rangle_*, \omega,q)\longrightarrow    (\tilde R\langle  S\rangle_*, \omega,q+2l).  
   \label{eq-cm2}
   \end{eqnarray}
     Consider   the  polynomial   algebras 
  \begin{eqnarray*}
\wedge^{2*} \Big( \frac{\partial}{\partial  s}\mid  s\in  S \Big)&=&  \bigoplus_{l\geq  0}   \wedge^{2l} \Big(\frac{\partial}{\partial  s}\mid  s\in  S \Big),\\     
\wedge^{2*}  ( ds \mid  s\in  S ) &=& \bigoplus_{l\geq  0}   \wedge^{2l} ( ds \mid  s\in  S ).     
 \end{eqnarray*}
 On the  other  hand,    consider  the  $R$-modules 
 \begin{eqnarray*}
   C((\tilde R\langle  S\rangle_*, \alpha,q),   (\tilde R\langle  S\rangle_*, \alpha,q-2l)),~~~
   C((\tilde R\langle  S\rangle_*, \omega,q),   (\tilde R\langle  S\rangle_*, \omega,q+2l)).   
   \end{eqnarray*} 
 Take  the  direct  sums   
 \begin{eqnarray*}
  C_-(\tilde R\langle  S\rangle_*, \alpha,q)&=&\bigoplus_{l=0}^\infty   C((\tilde R\langle  S\rangle_*, \alpha,q),   (\tilde R\langle  S\rangle_*, \alpha,q-2l)),\\
  C_+(\tilde R\langle  S\rangle_*, \omega,q)&=&\bigoplus_{l=0}^\infty   C((\tilde R\langle  S\rangle_*, \omega,q),   (\tilde R\langle  S\rangle_*, \omega,q+2l)).
   \end{eqnarray*}
  For  any  $l_1,l_2\in \mathbb{N}$  and    any  
  \begin{eqnarray*}
 & \varphi_1\in  C((\tilde R\langle  S\rangle_*, \alpha,q-2l_2),   (\tilde R\langle  S\rangle_*, \alpha,q-2({l_1}+l_2))),\\
&  \varphi_2\in  C((\tilde R\langle  S\rangle_*, \alpha,q),   (\tilde R\langle  S\rangle_*, \alpha,q-2l_2)), 
  \end{eqnarray*}   
    define  the  multiplication  of  $\varphi_1$  and  $\varphi_2$  as  the    composition   
  \begin{eqnarray}\label{eq-mmvdaq1}
  \varphi_1\circ\varphi_2  \in   C((\tilde R\langle  S\rangle_*, \alpha,q),   (\tilde R\langle  S\rangle_*, \alpha,q-2{(l_1+l_2)}))    
  \end{eqnarray}
  given  by  
  \begin{eqnarray*}
 (\tilde R\langle  S\rangle_*, \alpha,q) \overset{\varphi_2}{\longrightarrow }   (\tilde R\langle  S\rangle_*, \alpha,q-2l_2) \overset{\varphi_1}{\longrightarrow }   (\tilde R\langle  S\rangle_*, \alpha,q-2(l_1+l_2)).   
  \end{eqnarray*}
 % Extend  the  multiplication in  (\ref{eq-mmvdaq1})  bilinearly  over   $R$.   We  obtain  a  multiplication    
 %   \begin{eqnarray*}
% \circ:~~       (C_-(\tilde R\langle  S\rangle_*, \alpha,q),  C_-(\tilde R\langle  S\rangle_*, \alpha,q))\longrightarrow    C_-(\tilde R\langle  S\rangle_*, \alpha,q).  
%  \end{eqnarray*}
  Similarly,  for    any   
  \begin{eqnarray*}
  &\psi_1\in  C((\tilde R\langle  S\rangle_*, \omega,q+2l_2),   (\tilde R\langle  S\rangle_*, \omega,q+2(l_1+{l_2}))),\\
  &\psi_2\in  C((\tilde R\langle  S\rangle_*, \omega,q),   (\tilde R\langle  S\rangle_*, \omega,q+2{l_2})), 
  \end{eqnarray*}  %$i=1,2$,  
  define   the   multiplication  of   $\psi_1$  and  $\psi_2$  as  the    composition  
    \begin{eqnarray}\label{eq-mmvdaq2}
  \psi_1\circ\psi_2  \in   C((\tilde R\langle  S\rangle_*, \omega,q),   (\tilde R\langle  S\rangle_*, \omega,q+2{(l_1+l_2)}))      
  \end{eqnarray}
  given  by 
  \begin{eqnarray*}
 (\tilde R\langle  S\rangle_*, \omega,q) \overset{\psi_2}{\longrightarrow } (\tilde R\langle  S\rangle_*, \omega,q+2l_2) \overset{\psi_1}{\longrightarrow } (\tilde R\langle  S\rangle_*, \omega,q+2(l_1+l_2)).  
  \end{eqnarray*}
  % Extend  the  multiplication (\ref{eq-mmvdaq2})  bilinearly  over   $R$.   We  obtain  a  multiplication  from  
  %  \begin{eqnarray*}
% \circ:~~       (C_+(\tilde R\langle  S\rangle_*, \omega,q),  C_+(\tilde R\langle  S\rangle_*, \omega,q))\longrightarrow    C_+(\tilde R\langle  S\rangle_*, \omega,q).  
%  \end{eqnarray*}

 \begin{theorem}\label{pr-bmaq17}
(\ref{eq-cm1})  gives  a   homomorphism  from  $\wedge^{2*}   ({\partial}/{\partial  s}\mid  s\in  S  )$   to   $ C_-(\tilde R\langle  S\rangle_*, \alpha,q)$  and  (\ref{eq-cm2})  gives  a  homomorphism  from  $\wedge^{2*}   ( d s  \mid  s\in  S  )$   to   $ C_+(\tilde R\langle  S\rangle_*, \omega,q)$. 
 \end{theorem}

 \begin{proof}
 By  (\ref{eq-cm1}),  we  have  a  homomorphism  of   graded $R$-modules     from $\wedge^{2*}   ({\partial}/{\partial  s}\mid  s\in  S  )$   to   $ C_-(\tilde R\langle  S\rangle_*, \alpha,q)$.     
  For any  $\beta_i\in  \wedge^{2l_i}  ({\partial}/{\partial  s}\mid  s\in  S )$,   $i=1,2$,  
 $\beta_1\wedge \beta_2=\beta_2\wedge\beta_1$  induces  a  chain  map  such  that  the  following  diagram  commutes 
   \begin{eqnarray*}
 \xymatrix{
 (\tilde R\langle  S\rangle_*, \alpha,q)\ar[rr]^-{\beta_1} \ar[dd]_-{\beta_2} \ar[rrdd] ^-{~~~\beta_1\wedge \beta_2=\beta_2\wedge\beta_1}&&      (\tilde R\langle  S\rangle_*, \alpha,q-2l_1)\ar[dd]^-{\beta_2}\\
 \\
(\tilde R\langle  S\rangle_*, \alpha,q-2l_2) \ar[rr]^-{\beta_1} &&     (\tilde R\langle  S\rangle_*, \alpha,q-2(l_1+l_2)).   
 }
 \end{eqnarray*}
Thus   the  homomorphism  of   graded $R$-modules     from $\wedge^{2*}   ({\partial}/{\partial  s}\mid  s\in  S  )$   to   $ C_-(\tilde R\langle  S\rangle_*, \alpha,q)$  preserves  the  multiplication.

Similarly,  by  (\ref{eq-cm2}),  we  have  a  homomorphism  of   graded $R$-modules     from $\wedge^{2*}   (ds\mid  s\in  S  )$   to   $ C_+(\tilde R\langle  S\rangle_*, \omega,q)$.   
  For any  $\mu_i\in  \wedge^{2l_i}  (d  s \mid  s\in  S )$,   $i=1,2$,  
 $\mu_1\wedge \mu_2=\mu_2\wedge\mu_1$  induces  a  chain  map  such  that  the  following  diagram  commutes 
   \begin{eqnarray*}
 \xymatrix{
 (\tilde R\langle  S\rangle_*, \omega,q)\ar[rr]^-{\mu_1} \ar[dd]_-{\mu_2} \ar[rrdd] ^-{~~~\mu_1\wedge \mu_2=\mu_2\wedge\mu_1}&&      (\tilde R\langle  S\rangle_*, \omega,q+2l_1)\ar[dd]^-{\mu_2}\\
 \\
(\tilde R\langle  S\rangle_*, \omega,q+2l_2) \ar[rr]^-{\mu_1} &&     (\tilde R\langle  S\rangle_*, \omega,q+2(l_1+l_2)).   
 }
 \end{eqnarray*}
Thus the   homomorphism  of   graded $R$-modules     from $\wedge^{2*}   (  ds \mid  s\in  S  )$   to   $ C_+(\tilde R\langle  S\rangle_*, \omega,q)$  preserves  the  multiplication.    
  \end{proof}

  Let $R=\mathbb{C}$ be  the complex  numbers.     Then  $\tilde {\mathbb{C}}\langle S\rangle$  
  has  a  canonical  inner  product 
  such  that  
  \begin{eqnarray*}
  \Big\langle \sum   z_{s_0,\ldots,  s_n}  s_0\ldots  s_n,      \sum  {z'}_{ t_0,\ldots,  t_n}  t_0\ldots  t_n \Big\rangle = \sum   z_{s_0,\ldots,  s_n}  \bar  {z'}_{ t_0,\ldots,  t_n}  \prod_{i=0}^n   \delta(s_i,t_i)
  \end{eqnarray*}
 and  $\mathbb{C}\langle  S\rangle_n$  and    $\mathbb{C}\langle  S\rangle_m$  
 are  orthogonal  for  $n\neq  m$.   
 Two   linear  maps  $A$  and  $B$  from  $\tilde{\mathbb{C}}\langle S\rangle$   to  itself  are        adjoint   if
 $\langle  A\xi,\eta\rangle  = \overline{\langle \xi,  B\eta\rangle}$  for  any  $\xi,\eta\in  \tilde{\mathbb{C}}\langle S\rangle$.  
 
 \begin{lemma}\label{le-m018a}
 For  any  $n\in  \mathbb{N}$,  
 \begin{eqnarray*}
\alpha=\sum_{s_1, \ldots, s_n\in S} z_{s_1, \ldots, s_n} \frac{\partial}{\partial s_1} \wedge \cdots \wedge\frac{\partial}{\partial  s_n}
\end{eqnarray*}
and 
\begin{eqnarray*} 
\omega={\rm sgn}(n)\sum_{s_1, \ldots, s_n\in  S} \bar z_{s_1, \ldots,  s_n} d s_1\wedge  \cdots \wedge  d s_n,
\end{eqnarray*}
 where $z_{s_1, \ldots, s_n}\in\mathbb{C}
$,   ${\rm sgn}(n)=1$  if $n\equiv 0,1$  modulo  $4$  and ${\rm sgn}(n)=-1$  if $n\equiv 2,3$  modulo $4$,  are  adjoint.  In  particular,  $\partial/\partial s$  and  $ds$  are  adjoint.   
 \end{lemma}
 
 \begin{proof}
 The  proof  is  an  analogue  of  \cite[ Eq.  (3.5),   Lemma~3.2  and  Proposition~3.1]{camb2023}. 
 \end{proof}
 
 \begin{proposition}\label{pr-mvxza1}
 Let  $n=2k+1$.  
 Let  $\alpha$  and  $\omega$  be  given  in  Lemma~\ref{le-m018a}.  
 Let  $\beta_*$  be  the  Betti numbers  of  the  chain  complex  $(\tilde{\mathbb{C}}\langle  S\rangle_*, \alpha,q)$  and  
 let $\beta^*$  be the  Betti  numbers  of   the  chain  complex   $(\tilde{\mathbb{C}}\langle  S\rangle_*, \omega,q)$.  
  Then   for  any   $p\in \mathbb{N}$,  
$
 \beta_p ( \tilde{\mathbb{C}}\langle  S\rangle_*, \alpha,q)=\beta^p (\tilde{\mathbb{C}}\langle  S\rangle_*, \omega,q)$.   
 In  particular,  $
 \beta_p ( \tilde{\mathbb{C}}\langle  S\rangle_*, \partial/\partial s,q)=\beta^p (\tilde{\mathbb{C}}\langle  S\rangle_*,  ds,q)$.   
 \end{proposition}
 
 \begin{proof}
 
 Since %by  Lemma~\ref{le-m018a},
    $\alpha$  and  $\omega$  are  adjoint,  
 for  the  chain  complexes  $(\tilde{\mathbb{C}}\langle  S\rangle_*, \alpha,q)$  and    $(\tilde{\mathbb{C}}\langle  S\rangle_*, \omega,q)$  we  have 
 \begin{eqnarray*}
 & {\rm  Ker} ( \alpha_{p+1})=    {\rm  Im} ( \omega_{p}) ^\perp\cong {\rm  Coker}(\omega_{p}), \\
   &\overline{ {\rm  Im} ( \omega_{p}) }={\rm  Ker} ( \alpha_{p+1})^\perp\cong  {\rm  Im} ( \alpha_{p+1}), \\
 &{\rm  Ker} ( \omega_{p})=    {\rm  Im} ( \alpha_{p+1}) ^\perp\cong {\rm  Coker}(\alpha_{p+1}), \\
  &\overline{ {\rm  Im} ( \alpha_{p+1}) }={\rm  Ker} ( \omega_{p})^\perp\cong {\rm   Im} ( \omega_{p}).  
 \end{eqnarray*}
 It  follows  that  the  Betti  numbers  of     $(\tilde{\mathbb{C}}\langle  S\rangle_*, \alpha,q)$  and 
 $(\tilde{\mathbb{C}}\langle  S\rangle_*, \omega,q)$  satisfy  
 \begin{eqnarray*}
 \beta_p ( \tilde{\mathbb{C}}\langle  S\rangle_*, \alpha,q)&=& \dim{\rm  Ker} ( \alpha_{p})-\dim  {\rm  Im} ( \alpha_{p+1})\\
 &=& \dim {\rm  Coker} ( \omega_{p-1}) - \dim   {\rm   Im} ( \omega_{p}), \\
\beta^p( \tilde{\mathbb{C}}\langle  S\rangle_*, \omega,q)  &=& \dim  {\rm  Ker} ( \omega_{p})-\dim   {\rm  Im} ( \omega_{p-1})\\
&=&\dim {\rm  Coker}(\alpha_{p+1})  - \dim  {\rm  Im} ( \alpha_{p}). 
 \end{eqnarray*}
 With  the  help  that 
 \begin{eqnarray*}
  \dim {\rm  Ker} (\alpha_p) +  \dim {\rm   Im} (\alpha_p) 
 & =&
  \dim {\rm  Ker} (\omega_p) +  \dim {\rm   Im} (\omega_p)  \\
& =& \dim {\rm   Coker} (\alpha_{p+1}) +  \dim {\rm   Im} (\alpha_{p+1}) \\
&=&\dim {\rm   Coker} (\omega_{p-1}) +  \dim {\rm   Im} (\omega_{p-1})\\
& =&
  \dim \tilde{\mathbb{C}}\langle  S\rangle_{p ( 2k+1)+q},  
 \end{eqnarray*}
 we  obtain 
\begin{eqnarray}\label{eq-0aq109}
 \beta_p (\tilde{\mathbb{C}}\langle  S\rangle_*, \alpha,q)=\beta^p (\tilde{\mathbb{C}}\langle  S\rangle_*, \omega,q).   \end{eqnarray}
 Let  $\alpha=\partial/\partial  s$  and  $\omega= ds$  in  (\ref{eq-0aq109}).    
  We  obtain  $
 \beta_p ( \tilde{\mathbb{C}}\langle  S\rangle_*, \partial/\partial s,q)=\beta^p (\tilde{\mathbb{C}}\langle  S\rangle_*,  ds,q)$.   
 \end{proof}

 \section{Sub-chain  complexes   and  quotient chain  complexes  for  the  free  algebra}\label{s4}

% Let   $X$  be  a  subset  of  $\langle  S\rangle$. 
%   Let  $X_n$  be  the  set  consisting  of  the  elements  
% of  the  form  $s_0\ldots  s_n$  in  $X$.  
%   Then  $X_n\subseteq  \langle  S\rangle_n$   and  $X=\coprod_{n\geq  0}  X_n$.  
% Let  $R(  X)$  be  the  free  $R$-module  spanned  by  $X$.   
% Let  $R_n(X)$   be  the  free  $R$-module  spanned  by  $X_n$.  
 %Then  $R(  X)  =\oplus_{n\geq  0}    R_n( X)$
 % is  a  graded  sub-$R$-module  of  $R\langle  S\rangle=\oplus_{n=  0}^\infty    R\langle  S\rangle_n$.  

 For  any  $s_0 s_1\ldots s_n\in  \langle  S\rangle_n $,  we  call it       {\it   cyclic}  if    there exist   $0\leq i<j\leq n$  such that   $s_j= s_i$  and  call it   {\it  acyclic}  if  it  is  not  cyclic.    
 Let  $\mathcal{C}\langle  S\rangle_n$  be   the free  $R$-module   spanned by all the  cyclic    elements   in $ \langle  S\rangle_n $. 
 Let  $\mathcal{C}\langle  S\rangle =\oplus_{n=0}^\infty    \mathcal{C}\langle  S\rangle_n$. 
 Let  $\mathcal{D}\langle  S\rangle_n$  be   the free  $R$-module   spanned by all the  acyclic    elements   in $ \langle  S\rangle_n $. 
 Let  $\mathcal{D}\langle  S\rangle =\oplus_{n=0}^\infty    \mathcal{D}\langle  S\rangle_n$. 
 Then 
 \begin{eqnarray}\label{eq-00mca1}
 R\langle  S\rangle=\mathcal{C}\langle  S\rangle \oplus \mathcal{D}\langle  S\rangle
 .  
 \end{eqnarray}
The   bilinear  map  from  $(R\langle  S\rangle_n,  R\langle  S\rangle_m)$  to   $ R\langle  S\rangle_{n+m}$  
 induced  by  the   multiplication  of  $\langle  S\rangle $  sends  
 $(\mathcal{C}\langle  S\rangle _n,  R\langle  S\rangle_m) $   and  
 $(  R\langle  S\rangle_n,\mathcal{C}\langle  S\rangle _m) $
 to  $\mathcal{C}\langle  S\rangle_{n+m}  $.

 Suppose  $S$  has    a  total  order  $\prec$.  
 Let  $(S)$   be  the  free  abelian semi-group  generated  by  $S$.   
 The  abelianization  
  of   $\langle  S\rangle $  is  a  canonical   epimorphism  $\pi:  \langle  S\rangle \longrightarrow   (S)$. 
  Let  $(S)_n= \pi\langle  S\rangle_n$.    
 Let  $R(S)_n$  be  the   free  $R$-module   spanned  by  $(S)_n$.  
 The  abelianization   of   $\langle  S\rangle_n $  induces  a  canonical    epimorphism   of  $R$-modules  
 $R(\pi):   R\langle  S\rangle_n\longrightarrow  R(S)_n$
  sending   
  $s_{g(0)} s_{g(1)}\ldots  s_{g(n)}$  to  $ {{\rm  sgn}(g)}  \pi(s_0s_1\ldots  s_n)$
   for  any  $s_0\prec s_1\prec\cdots \prec s_n$  and  any  $g\in \Sigma_{n+1}$,  where 
   $\Sigma_{n+1}$  is  the  permutation  group  on  $0,1,\ldots,  n$  and  ${\rm  sgn}(g)=1$  if  
   $g$  is  an  even  permutation and    ${\rm  sgn}(g)=-1$  if  $g$  is  an  odd  permutation.  
      Let  $\mathcal{C}(S)_n=  R(\pi)(\mathcal{C}\langle  S\rangle_n)$  and 
  $\mathcal{D}(S)_n= R(\pi)(\mathcal{D}\langle  S\rangle)_n$.  
  Then  
  $R(S)_n= \mathcal{C}(S)_n\oplus\mathcal{D}(S)_n$.  
    The  multiplication  of  $(S)$   extends  to  be  a  bilinear  map      
     from  
    %      $(R( S),  R( S)  )$ to  $R(  S)$,   
%sending 
     $(R( S)_n,  R( S_m)  )$ to  $R(  S)_{n+m}$,     sending  both  
     $(\mathcal{C}( S)_n,  R(  S)_m) $   and  
 $(  R(  S)_n,\mathcal{C}( S) _m) $
 to  $\mathcal{C}(  S)_{n+m}  $.  %,  such  that  the  diagram  commutes 
% \begin{eqnarray*}
% \xymatrix{
 %(R\langle  S\rangle,  R\langle  S\rangle)  \ar[r]  \ar[d]_{(\pi,\pi)}   &  R\langle  S\rangle \ar[d]^\pi \\
 %  (R(  S),  R(S))  \ar[r]     &  R( S).    
 %}
% \end{eqnarray*}

     We  call  an  acyclic  element     $s_0s_1  \ldots s_n\in  \langle  S\rangle_n$  
   {\it  non-simplicial}    if there exist integers  $0\leq i<j\leq n$  such that   $s_j\prec  s_i$    (cf.  \cite[Definition~4.1]{camb2023})  and  call  an  acyclic  element     $s_0s_1  \ldots s_n\in  \langle  S\rangle_n$  {\it  simplicial} otherwise.   
    Let $\mathcal{E}\langle   S\rangle _n$   be   the  free   $R$-module   spanned by all the  non-simplicial  acyclic   elements   in    $\langle  S\rangle_n$.       Let $\mathcal{F}\langle   S\rangle _n$   be   the  free   $R$-module   spanned by all the   simplicial  acyclic   elements   in    $\langle  S\rangle_n$.  
    Let  $\mathcal{E}\langle   S\rangle=\oplus_{n\geq  0}     \mathcal{E}\langle   S\rangle_n$
    and   $\mathcal{F}\langle   S\rangle=\oplus_{n\geq  0}     \mathcal{F}\langle   S\rangle_n$. 
     Then   $\mathcal{E}\langle   S\rangle$  and   $\mathcal{F}\langle   S\rangle$   are     graded  sub-$R$-modules  of 
     $\mathcal{D}\langle   S\rangle$  such that  
     \begin{eqnarray}\label{eq-00mca2}
      \mathcal{D}\langle   S\rangle= \mathcal{E}\langle   S\rangle\oplus\mathcal{F}\langle   S\rangle. 
     \end{eqnarray}
                 For  each  $\pi(s_0  \ldots s_n) \in  (S)_n$  where  $s_0 \ldots s_n\in \langle  S\rangle_n$  is  acyclic,  
      there  is  a  unique  simplicial  representative  for  
  $ \pi(s_0  \ldots s_n)$,  
  say  $s_0  \ldots s_n$  such  that  $s_0 \prec  \cdots\prec  s_n$.  
    Therefore,  as graded  $R$-modules, 
  $\mathcal{F}\langle  S\rangle   \cong   D(S) $  and 
    $\mathcal{E}\langle   S\rangle \cong  {\rm  Ker}(\pi:  D\langle  S\rangle\longrightarrow  D(S))$.

  Let  $s_0 \ldots  s_n$  and  $t_0 \ldots  t_m$  be  
  simplicial  acyclic  elements  in  $\langle  S\rangle$.  
  Define  
  $(s_0 \ldots  s_n)*(t_0 \ldots  t_m)$  
  to  be  the  simplicial  representative  in  $\mathcal{F}\langle  S\rangle$  for 
  $ \pi(s_0 \ldots  s_n t_0 \ldots  t_m) $
  if  $s_0 \ldots  s_n t_0\ldots  t_m$  is   acyclic  and  
    to  be  $0$  otherwise.  
  Extend  $*$  to  be  a  bilinear  map  
  from  $(\mathcal{F}\langle   S\rangle, \mathcal{F}\langle   S\rangle)$  to  $\mathcal{F}\langle   S\rangle$  sending  $(\mathcal{F}\langle   S\rangle_n, \mathcal{F}\langle   S\rangle_m)$  to  $\mathcal{F}\langle   S\rangle_{n+m}$.

      Let  $\mathcal{F}\langle  S\rangle_{-1}= R$  and  let  $\tilde {\mathcal{F}}\langle  S\rangle =\mathcal{F}\langle S\rangle  \oplus  \mathcal{F}\langle S\rangle_{-1}$.  
  Let  $s\in  S$.   
Restricting  (\ref{eq-2.za1})  to 
 %$\mathcal{D}\langle  S\rangle_n$  and  
 $\mathcal{F}\langle  S\rangle_n$,  %  respectively,    
 we    have     an   induced    $R$-linear map   %of  free  $R$-modules 
\begin{eqnarray}\label{eq-wwwvvva1}
% \overrightarrow{\frac{\partial} {\partial  s}}:  && \mathcal{D}\langle  S\rangle_n\longrightarrow  \prod_{n+1}\mathcal{D}\langle  S\rangle_{n-1}, 
% \\
 \overrightarrow{\frac{\partial}{\partial  s}}:  &&  { \mathcal{F}}\langle  S\rangle_n\longrightarrow  \prod_{n+1}  {\mathcal{F}}\langle  S\rangle_{n-1}.
\end{eqnarray}   
Restricting  (\ref{eq-xvca3})  to  %$\mathcal{C}\langle  S\rangle_n$  and  
  $\mathcal{C}\langle  S\rangle_n\oplus \mathcal{E}\langle  S\rangle_n$,  %  respectively,  
 we    have     an   induced    $R$-linear map  % of  free  $R$-modules 
\begin{eqnarray}\label{eq-0bmdq3}
% \overrightarrow{ds}:  && \mathcal{C}\langle  S\rangle_n\longrightarrow  \prod_{n+2}\mathcal{C}\langle  S\rangle_{n+1},
% \\
 \overrightarrow{ds}:  &&  \mathcal{C}\langle  S\rangle_n\oplus\mathcal{E}\langle  S\rangle_n\longrightarrow  \prod_{n+2} (\mathcal{C}\langle  S\rangle_{n+1}\oplus \mathcal{E}\langle  S\rangle_{n+1})
\end{eqnarray}   
which,  with  the  help  of   (\ref{eq-xvca3}),  (\ref{eq-00mca1})  and  (\ref{eq-00mca2}),     induces a   quotient  $R$-linear  map   %of  free  $R$-modules 
\begin{eqnarray}\label{eq-wwwvvva2}
 %\overrightarrow{ds}:  && \mathcal{D}\langle  S\rangle_n\longrightarrow  \prod_{n+2}\mathcal{D}\langle  S\rangle_{n+1},
 %\\
 \overrightarrow{ds}:  &&   {\mathcal{F}}\langle  S\rangle_n\longrightarrow  \prod_{n+2}    {\mathcal{F}}\langle  S\rangle_{n+1}. 
\end{eqnarray} 
Explicitly,    the  map $d_is:  \mathcal{F}\langle  S\rangle_n\longrightarrow\mathcal{F}\langle  S\rangle_{n+1} $  in   $ {\overrightarrow{ds}}=  (d_0s,\ldots,d_{n+1}s)$  in   (\ref{eq-wwwvvva2})  is  given  by 
\footnote[2]{We  point  out  that  there  is  a  detailed minor  mistake  in     \cite[  Sec.  6.1]{jktr2023}.  
The  last  two  formulae  in   \cite[p.  26]{jktr2023}  might  be  revised  according  to  (\ref{eq-ada2b}). }
\begin{eqnarray}\label{eq-ada2b}
d_is (s_0 \ldots  s_{n})=
\begin{cases}
 (-1)^i  s_0 \ldots s_{i-1} s s_i  \ldots s_{n}   &{\rm~if~} s_{i-1}\prec  s\prec  s_i   \\
 0    &{\rm   ~ otherwise~}    
 \end{cases}
\end{eqnarray}
where   $s_0 \prec \cdots  \prec  s_{n}$  and        $0\leq  i\leq  n+1$.   Here  if  $i=0$,  then    $s\prec  s_0$   and   the  righthand side  of  (\ref{eq-ada2b})  
is   $s s_0 \ldots s_{n}$;      and  if   $i=n+1$,  then  $s_n\prec  s$  and   the  righthand side  of  (\ref{eq-ada2b})  
is   $(-1)^{n+1}   s_0 \ldots  s_{n}   s$.

By  (\ref{eq-wwwvvva1}),  
${\partial}/{\partial  s}  = \sum_{i=0}^n   {\partial_i}/{\partial  s}$  
is  an  $R$-linear  map   from  
$ {\mathcal{F}}\langle  S\rangle_n$  to  $  {\mathcal{F}}\langle  S\rangle_{n-1}$.    
  By  (\ref{eq-wwwvvva2}), 
     $ds  = \sum_{i=0}^{n+1}  d_i s$,   
     where   at  most  one  $ d_i s$  is  nonzero  in  the  summation,  
      is  an   $R$-linear  map   from  
$ {\mathcal{F}}\langle  S\rangle_n$  to  $ {\mathcal{F}}\langle  S\rangle_{n+1}$.

 Let  $k\in  \mathbb{N}$.  
 Let  $\alpha\in   \wedge^{2k+1} ( {\partial}/{\partial  s}\mid  s\in  S  )$  
 and    $\omega\in  \wedge^{2k+1}  (  d s \mid  s\in  S  )$.  
    Let  $q\in \mathbb{N}$.  
    We  have  chain  complexes 
   \begin{eqnarray*}
  &&  \cdots   \overset{\alpha_{p+2}}{\longrightarrow}    {\mathcal{F}}\langle   S\rangle_{(p+1)(2k+1)+q}  \overset{\alpha_{p+1}}{\longrightarrow}  {\mathcal{F}}\langle   S\rangle_{p ( 2k+1)+q}\overset{\alpha_p}{\longrightarrow}  {\mathcal{F}}\langle   S\rangle_{(p-1)(2k+1)+q}\overset{\alpha_{p-1}}{\longrightarrow}   \cdots,\\
 &&  \cdots   \overset{\omega_{p-2}}{\longrightarrow}    {\mathcal{F}}\langle   S\rangle_{(p-1)(2k+1)+q}  \overset{\omega_{p-1}}{\longrightarrow}  {\mathcal{F}}\langle   S\rangle_{p ( 2k+1)+q}\overset{\omega_p}{\longrightarrow}  {\mathcal{F}}\langle   S\rangle_{(p+1)(2k+1)+q}\overset{\omega_{p+1}}{\longrightarrow}   \cdots
   \end{eqnarray*}
 denoted  as   $(\tilde {\mathcal{F}}\langle  S\rangle_*, \alpha,q)$  and    $(\tilde {\mathcal{F}}\langle  S\rangle_*, \omega,q)$    respectively.  
  By   (\ref{eq-wwwvvva1}),    $(\tilde {\mathcal{F}}\langle  S\rangle_*, \alpha,q)$   
 is  a  sub-chain  complex  of   $(\tilde  R\langle  S\rangle, \alpha,q)$ 
  and   by (\ref{eq-wwwvvva2}),   $(\tilde {\mathcal{F}}\langle  S\rangle_*, \omega,q)$    
  is  a  quotient  complex  of    $(\tilde  R\langle  S\rangle, \omega,q)$ 
  moduling   its  sub-chain  complex  
  $ (\mathcal{C}\langle  S\rangle_* \oplus\mathcal{E}\langle  S\rangle_*, \omega,  q)$.  
  Let   $l\in\mathbb{N}$.     Let  $\beta\in   \wedge^{2l}  ({\partial}/{\partial  s}\mid  s\in  S )$  and    $\mu\in  \wedge^{2l}  (  d s \mid  s\in  S  )$.   
  Restricted  to    $(\tilde {\mathcal{F}}\langle  S\rangle_*, \alpha,q)$,   (\ref{eq-cm1})   induces   chain  map  
    \begin{eqnarray} \label{eq-cm8d}
 \beta:  &&(\tilde {\mathcal{F}}\langle  S\rangle_*, \alpha,q)\longrightarrow    (\tilde {\mathcal{F}}\langle  S\rangle_*, \alpha,q-2l).
   \end{eqnarray}
   Restricted  to   $ (\mathcal{C}\langle  S\rangle_* \oplus\mathcal{E}\langle  S\rangle_*, \omega,  q)$,   (\ref{eq-cm2})   induces   chain  map  
    \begin{eqnarray} \label{eq-cm9d}
 \mu:  && (\mathcal{C}\langle  S\rangle_* \oplus\mathcal{E}\langle  S\rangle_*, \omega,  q)\longrightarrow     (\mathcal{C}\langle  S\rangle_* \oplus\mathcal{E}\langle  S\rangle_*, \omega,  q+2l).
   \end{eqnarray}
   The  chain  maps   (\ref{eq-cm2})  and  (\ref{eq-cm9d})   induce   a  quotient  chain  map  
  \begin{eqnarray} \label{eq-cm10d}
 \mu:  && (\tilde {\mathcal{F}}\langle  S\rangle_*, \omega,  q)\longrightarrow     (\tilde {\mathcal{F}}\langle  S\rangle_*, \omega,  q+2l).
   \end{eqnarray}
   On  the  other  hand,  
  consider  the  $R$-modules 
 \begin{eqnarray*}
   C((\tilde {\mathcal{F}}\langle  S\rangle_*, \alpha,q),   (\tilde {\mathcal{F}}\langle  S\rangle_*, \alpha,q-2l)), 
~~~  C((\tilde {\mathcal{F}}\langle  S\rangle_*, \omega,q),   (\tilde {\mathcal{F}}\langle  S\rangle_*, \omega,q+2l)).   
   \end{eqnarray*} 
 Take  the  direct  sums   
 \begin{eqnarray*}
  C_-(\tilde {\mathcal{F}}\langle  S\rangle_*, \alpha,q)&=&\bigoplus_{l=0}^\infty   C((\tilde {\mathcal{F}}\langle  S\rangle_*, \alpha,q),   (\tilde {\mathcal{F}}\langle  S\rangle_*, \alpha,q-2l)),\\
  C_+(\tilde {\mathcal{F}}\langle  S\rangle_*, \omega,q)&=&\bigoplus_{l=0}^\infty   C((\tilde {\mathcal{F}}\langle  S\rangle_*, \omega,q),   (\tilde {\mathcal{F}}\langle  S\rangle_*, \omega,q+2l)).
   \end{eqnarray*}
Similar with  (\ref{eq-mmvdaq1}),   define  the  multiplication   $\varphi_1\circ\varphi_2$   for  any   
\begin{eqnarray*}
&\varphi_1\in  C((\tilde {\mathcal{F}}\langle  S\rangle_*, \alpha,q-2l_2),   (\tilde {\mathcal{F}}\langle  S\rangle_*, \alpha,q-2({l_1}+l_2))),\\
&\varphi_2\in  C((\tilde {\mathcal{F}}\langle  S\rangle_*, \alpha,q),   (\tilde {\mathcal{F}}\langle  S\rangle_*, \alpha,q-2l_2)).  
\end{eqnarray*}  
Similar  with  (\ref{eq-mmvdaq2}),  define  the  multiplication  $\psi_1\circ\psi_2$   
for    any   
\begin{eqnarray*}
&\psi_1\in  C((\tilde {\mathcal{F}}\langle  S\rangle_*, \omega,q+2l_2),   (\tilde {\mathcal{F}}\langle  S\rangle_*, \omega,q+2(l_1+{l_2}))),\\
&\psi_2\in  C((\tilde {\mathcal{F}}\langle  S\rangle_*, \omega,q),   (\tilde {\mathcal{F}}\langle  S\rangle_*, \omega,q+2{l_2})).  
\end{eqnarray*}

 \begin{proposition}
 \label{pr-bmaq17-co}
(\ref{eq-cm8d})  induces  a   homomorphism  from  $\wedge^{2*}   ({\partial}/{\partial  s}\mid  s\in  S  )$   to   $ C_-(\tilde {\mathcal{F}}\langle  S\rangle_*, \alpha,q)$  and    (\ref{eq-cm10d})  induces    a  homomorphism  from  $\wedge^{2*}   ( d s  \mid  s\in  S  )$   to   $ C_+(\tilde {\mathcal{F}}\langle  S\rangle_*, \omega,q)$. 
 \end{proposition}
 
 \begin{proof}
 The proof  is  an  analogue  of  Theorem~\ref{pr-bmaq17}.  
 \end{proof}

Suppose  both  $S$  and  $S'$  are  sets  with  total  orders.   A  map  $f:  S\longrightarrow  S'$
 is  said to preserve  the  total  order  if  for  any  $s,t\in  S$  such  that  $s\prec t$,  
 we  have  $f(s)\prec  f(t)$.  
    Let  $f:  S\longrightarrow  S'$  be  a  map  preserving  the  total  order.   
  We  have  an  induced  graded  $R$-linear  map  
   $R(f):  R\langle  S\rangle  \longrightarrow  R\langle  S'\rangle  $  given  by
$ R(f)(s_0\ldots  s_n)=f(s_0)\ldots  f(s_n)
$   for  each  $s_0\ldots  s_n\in \langle  S\rangle$.   
  The   composition  of  the   restriction  of  $R(f)$  to    $ \mathcal{F}\langle  S\rangle$  and     
   the  canonical  projection    from  $R\langle  S'\rangle$  to  $\mathcal{D}\langle  S'\rangle$, 
     whose  image  is   in  $\mathcal{F}\langle  S'\rangle$,  
  gives  a   graded  $R$-linear  map 
 \begin{eqnarray}\label{eq-0mmbs6}
 \mathcal{F}(f): ~~~ \mathcal{F}\langle  S\rangle  \overset{R(f)}{\longrightarrow}  R\langle  S'\rangle   \twoheadrightarrow    \mathcal{F}\langle  S'\rangle.  
       \end{eqnarray}
            Let  
  $\wedge(f):  \wedge(\partial/\partial s\mid  s\in  S)\longrightarrow  \wedge(\partial/\partial s'\mid  s'\in  S')$  be  the  homomorphism  of  exterior  algebras  sending  
  $\partial/\partial s$  to  $\partial/\partial (f(s))$  for  any  $s\in  S$  and   
  let  
    $\wedge(f):  \wedge(d s\mid  s\in  S)\longrightarrow  \wedge(d s'\mid  s'\in  S')$  be  the  homomorphism  of  exterior  algebras  sending  
  $ds$  to  $d(f(s))$  for  any   $s\in  S$. 

  \begin{proposition}\label{pr-008.990}
Let  $f:  S\longrightarrow  S'$  be an  injective  map   preserving  the  total  order.  
\begin{enumerate}[(1)]
\item
For  any   $\alpha\in  \wedge(\partial/\partial  s\mid  s\in  S)$  and any  $\xi\in  \mathcal{F}\langle  S\rangle$, 
  \begin{eqnarray*}
 \mathcal{F}(f)( \alpha (\xi))=  \wedge(f)(\alpha)(\mathcal{F}(f)(\xi)); 
  \end{eqnarray*} 
  \item
  For  any  $\omega\in  \wedge(d  s\mid  s\in  S)$  and any  $\xi\in  \mathcal{F}\langle  S\rangle$, 
  \begin{eqnarray*}
 \mathcal{F}(f)( \omega (\xi))=  \wedge(f)(\omega)(\mathcal{F}(f)(\xi)).  
  \end{eqnarray*} 
  \end{enumerate}
  \end{proposition}

\begin{proof}
(1)   Let  $\alpha = \partial/ \partial  s $  and  $\xi=s_0s_1\ldots  s_n$  where  $s_0\prec  s_1\prec\cdots\prec  s_n$.  
Then  
 \begin{eqnarray*}
   \mathcal{F}(f)( (\partial/\partial s) (s_0\ldots  s_n))  
   &=&   \mathcal{F}(f)\Big(\sum_{i=0}^n  (-1)^i\delta(s,s_i)  s_0\ldots\widehat{s_i} \ldots  s_n\Big)\nonumber\\
&=&\sum_{i=0}^n  (-1)^i\delta( s,s_i)    f(s_0)\ldots\widehat{f(s_i)} \ldots  f(s_n) \nonumber\\
&=&\sum_{i=0}^n  (-1)^i\delta( f(s),f(s_i))    f(s_0)\ldots\widehat{f(s_i)} \ldots  f(s_n) \nonumber\\
&=& ({\partial}/{\partial  f(s)}) ( f( s_0 ) f( s_1 )  \ldots  f( s_n) )\nonumber\\
&=&\wedge(f)( \partial/\partial s )(\mathcal{F}(f)( (s_0\ldots  s_n))).  
\label{eq-9112}
\end{eqnarray*}
 In  general,  for  $\alpha=  \wedge^k(\partial/\partial  s\mid  s\in  S)$  and  $\xi\in  \mathcal{F}\langle  S\rangle_n$,  we  apply    an  induction  on  $k$.    By  the  $R$-linearity,   the  proof      reduces  to   the  special  case  $\alpha =\partial/\partial  s$  and  $\xi=s_0s_1\ldots  s_n$.

(2)   Let  $\omega =ds$     and  
  $\xi = s_0s_1\ldots  s_n$  where   $s_0\prec\cdots\prec  s_n$.  
  If  $s=s_i$  for  some  $0\leq  i\leq  n$,  then  
  $\omega(\xi)=\wedge(f)(\omega)(\mathcal{F}(f)(\xi))=0$.   Thus  without  loss  of  generality,  
  we  assume   $s_{i-1}\prec  s\prec  s_i $  for  some  $0\leq  i\leq  n+1$.   
  Then  
   \begin{eqnarray*}
  \mathcal{F}(f)( ds   (s_0s_1\ldots  s_n)) 
   &=&  \mathcal{F}(f) ((-1)^i  s_0 \ldots  s_{i-1} s s_i  \ldots s_{n} )\nonumber\\
&=&(-1)^i   f(s_0) \ldots  f(s_{i-1})  f(s)  f(s_i)  \ldots f(s_{n}) \nonumber\\
&=&  (d f(s))   (f(s_0)f(s_1)\ldots  f(s_n))\nonumber\\
   &=& \wedge(f)(ds)(\mathcal{F}(f)(s_0s_1\ldots  s_n)).   
  \label{eq-9812}
  \end{eqnarray*}
  In  general,  for  $\omega=  \wedge^k(d  s\mid  s\in  S)$  and  $\xi\in  \mathcal{F}\langle  S\rangle_n$,  we  apply    an  induction  on  $k$.    By  the  $R$-linearity,   the  proof      reduces  to   the  special  case  $\omega =ds$  and  $\xi=s_0s_1\ldots  s_n$. 
\end{proof}

     \begin{corollary}\label{co-008.3}
     Let  $f:  S\longrightarrow  S'$  be  an  injective map   preserving  the  total  order.   For  any  $\alpha\in   \wedge^{2k+1}(\partial/\partial s\mid  s\in  S)$, 
        there  is  an  induced   chain  map
       \begin{eqnarray*}
       \mathcal{F}(f):  ~~~(\tilde {\mathcal{F}}\langle  S\rangle_*, \alpha,q)\longrightarrow  (\tilde {\mathcal{F}}\langle  S'\rangle_*, \wedge(f)(\alpha),q). 
       \end{eqnarray*}
       Moreover,  for  any  $\beta\in   \wedge^{2l}(\partial/\partial s\mid  s\in  S)$,   the  diagram  commutes 
       \begin{eqnarray*}
       \xymatrix{
       (\tilde {\mathcal{F}}\langle  S\rangle_*, \alpha,q)\ar[r]^-{\mathcal{F}(f)} \ar[d]_-{\beta}&  (\tilde {\mathcal{F}}\langle  S'\rangle_*, \wedge(f)(\alpha),q)\ar[d]^-{\wedge(f)(\beta)}\\
            (\tilde {\mathcal{F}}\langle  S\rangle_*, \alpha,q-2l)\ar[r]^-{\mathcal{F}(f)}  &  (\tilde {\mathcal{F}}\langle  S'\rangle_*, \wedge(f)(\alpha),q-2l). 
       }
       \end{eqnarray*}
       \end{corollary}
       
       \begin{proof}
       The  proof  follows  from  Proposition~\ref{pr-008.990}~(1). 
       \end{proof}

 \begin{corollary}
 \label{co-mzx}
     Let  $f:  S\longrightarrow  S'$  be  an  injective map   preserving  the  total  order.   For  any  $\omega\in   \wedge^{2k+1}( d  s\mid  s\in  S)$, 
        there  is  an  induced   chain  map
       \begin{eqnarray*}
       \mathcal{F}(f):  ~~~(\tilde {\mathcal{F}}\langle  S\rangle_*, \omega,q)\longrightarrow  (\tilde {\mathcal{F}}\langle  S'\rangle_*, \wedge(f)(\omega),q). 
       \end{eqnarray*}
       Moreover,  for  any  $\mu\in   \wedge^{2l}(\partial/\partial s\mid  s\in  S)$,   the  diagram  commutes 
       \begin{eqnarray*}
       \xymatrix{
       (\tilde {\mathcal{F}}\langle  S\rangle_*, \omega,q)\ar[r]^-{\mathcal{F}(f)} \ar[d]_-{\mu}&  (\tilde {\mathcal{F}}\langle  S'\rangle_*, \wedge(f)(\omega),q)\ar[d]^-{\wedge(f)(\mu)}\\
            (\tilde {\mathcal{F}}\langle  S\rangle_*, \omega,q+2l)\ar[r]^-{\mathcal{F}(f)}  &  (\tilde {\mathcal{F}}\langle  S'\rangle_*, \wedge(f)(\omega),q+2l). 
       }
       \end{eqnarray*}
 \end{corollary}
 
  \begin{proof}
       The  proof  follows  from  Proposition~\ref{pr-008.990}~(2). 
       \end{proof}
 
  \section{Hypergraphs  and  their  traces}\label{s---5}

 Let  $S$  be  a  set  of  vertices.    %A   {\it  hypergraph}      on   $S$  is  a  collection  of     non-empty  finite  subsets   of  $S$.   
  Let  $[2^{S}]$  be  the  collection of all   the    finite  subsets  of  $S$.  
  We  call     a  family    $\mathcal{H} $   of   finite  subsets   of  $S$,   i.e.  a  subset  of  $[2^{S}]$,   
  an  {\it  augmented  hypergraph}  on  $S$.   
  An  augmented hypergraph      is  either a  hypergraph   
  or    a  hypergraph  together with  $\emptyset$.  
   An  element  of  $\mathcal{H}$  consisting  of  $n+1$  vertices,   $n\in \mathbb{N}$,   is  called  an  {\it  $n$-hyperedge}.    With   an   abuse  of  notation,  if  $\emptyset\in\mathcal{H}$,  then  we also  call  $\emptyset$   a  hyperedge  of  $\mathcal{H}$.

           We  define  an    {\it   augmented  simplicial complex}  $\mathcal{K}$  on  $S$  
  to  be  an  augmented  hypergraph such  that     any non-empty  subset   of  any  hyperedge  in  $\mathcal{K}$  is  still  a 
  hyperedge in  $\mathcal{K}$.   
   %An  augmented  simplicial  complex   is  either a  simplicial  complex    
 % or    a  simplicial  complex  together with  $\emptyset$. 
   An   $n$-hyperedge  in   $\mathcal{K}$  is     called   an   {\it   $n$-simplex}.  
    We  define  an   {\it    augmented   independence  hypergraph}   $\mathcal{L}$  on   $S$  
 to  be    an  augmented   hypergraph such  that    any  finite   superset  (whose  elements are from  $S$)  of    any  hyperedge  in  $\mathcal{L}$   is  still  a  hyperedge  in  $\mathcal{L}$.  
  %An  augmented  independence  hypergraph   is  either an   independence  hypergraph  (cf.  \cite[Definition~2.9]{camb2023})      
  % or    an  independence  hypergraph  together with  $\emptyset$. 
  Note  that     both   $ [2^S]$  and  $[2^S]\setminus \{\emptyset\}$  are   augmented   simplicial  complexes  as  well   as      
augmented   independence  
  hypergraphs.

  Let  $\mathcal{H}$  be  an  augmented   hypergraph  on  $S$.  
   For  any  subset  $T\subseteq  S$,  the  {\it   trace}  of  $\mathcal{H}$
          on  $T$  is an  augmented  hypergraph  on  $T$  
          \begin{eqnarray*}
          \mathcal{H}\mid_T  = 
          \begin{cases}
           \{\sigma\cap  T   \mid  \sigma\in \mathcal{H}{\rm~such~that~}\sigma\cap  T\neq\emptyset\},   &  
           {\rm~if~}  \emptyset\notin \mathcal{H},\\ 
            \{\sigma\cap  T   \mid  \sigma\in \mathcal{H}{\rm~such~that~}\sigma\cap  T\neq\emptyset\}\cup\{\emptyset\},  &
           {\rm  ~if~}\emptyset\in \mathcal{H}.  
           \end{cases} 
          \end{eqnarray*} 
  For  any  $\sigma\in \mathcal{H}$,  let  
 $
  \Delta\sigma=
  \{\tau\subseteq\sigma\mid \tau\neq\emptyset\} $   if  $\sigma\neq\emptyset$ 
  and  let    $ \Delta\sigma= \sigma $
   if  $\sigma=\emptyset$;   and   
     let   $\bar\Delta\sigma=\{\tau\supseteq  \sigma\mid  \tau\subseteq  S\}$.   
  Let  the {\it associated   augmented  simplicial complex}    
 $
 \Delta\mathcal{H}=\cup_{\sigma\in\mathcal{H}} \Delta\sigma
$ 
 to  be  the  smallest  augmented  simplicial complex  containing   $\mathcal{H}$,    
 the   {\it  lower-associated  augmented  simplicial complex}     
$
 \delta\mathcal{H}=\cup_{\Delta\sigma\subseteq \mathcal{H}} \{\sigma\}%\{\tau\in\Delta[S]\mid \Delta\tau\subseteq  \mathcal{H}\} \cup\mathcal{H}
$
 to  be  the  largest augmented  simplicial complex  contained  in   $\mathcal{H}$,    
  the  {\it  associated augmented  independence  hypergraph}   
  $%\begin{eqnarray*} 
 \bar  \Delta\mathcal{H}  =\cup_{\sigma\in\mathcal{H}}\bar \Delta\sigma %\{\tau\in\Delta[S]\mid \tau\supset  \sigma {\rm~for~some~} \sigma\in \mathcal{H}\}  \cup \mathcal{H}
   $ %\end{eqnarray*}  
  to  be    the   smallest augmented  independence hypergraph  containing   $\mathcal{H}$,  
  and  the    {\it  lower-associated  augmented  independence hypergraph} 
  $%\begin{eqnarray*}
  \bar \delta\mathcal{H}=\cup_{\bar\Delta\sigma\subseteq \mathcal{H}} \{\sigma\}%\{\tau\in\Delta[S]\mid \bar \Delta\tau\subseteq  \mathcal{H}\} \cup\mathcal{H}
  $   %\end{eqnarray*} 
  to  be  the  largest augmented  independence  hypergraph   contained  in   $\mathcal{H}$.  
  Note  that  $\emptyset\in \mathcal{H}$  iff  
  $\emptyset\in  \delta\mathcal{H}$  iff
  $\emptyset \in  \Delta\mathcal{H}$
   iff  
   $\emptyset \in \bar\Delta\mathcal{H}$    
   iff  $ \bar\Delta\mathcal{H}=  [2^S]$.  
    In  addition,  if  $S$  is  finite,  then  
    we  we  define   the {\it  global  complement} of  $\mathcal{H}$  as  the  augmented  hypergraph     
    %\begin{eqnarray*}
     $\gamma_S\mathcal{H}=    [2^S] \setminus   \mathcal{H} $      
    %\end{eqnarray*}
      and     the  {\it  local  complement} of  $\mathcal{H}$  as  the  augmented  hypergraph  
      $   % \begin{eqnarray*}
      \Gamma_S\mathcal{H}  = \{\Gamma_S(\sigma)\mid  \sigma\in\mathcal{H}\} 
      $   %\end{eqnarray*}
      where  
$\Gamma_S(\sigma)=  S\setminus \sigma$  for  any  $\sigma\in  [2^S]$.

Let  $\mathcal{H}$  and  $\mathcal{H}'$  be   augmented  hypergraphs  
  on  $S$.   The   intersection  $\mathcal{H}\cap\mathcal{H}'$  
and  the  union  
$\mathcal{H}\cup\mathcal{H}'$     are    augmented  hypergraphs  on  $S$.  
 In  particular,  (1)   for  any  augmented  simplicial  complexes  
 $\mathcal{K}$  and  $\mathcal{K}'$  on  $S$,  
 $\mathcal{K}\cap\mathcal{K}'$  and    
$\mathcal{K}\cup\mathcal{K}'$  are  augmented  
simplicial  complexes  on  $S$;       
 (2)   for  any  augmented  independence  hypergraphs    
 $\mathcal{L}$  and  $\mathcal{L}'$  on  $S$,  
 $\mathcal{L}\cap\mathcal{L}'$  and    
$\mathcal{L}\cup\mathcal{L}'$  are  augmented  
independence  hypergraphs  on  $S$.

Let    $S$  and  $S'$ be   disjoint.   
Let  $\mathcal{H} $     and    $\mathcal{H}' $  be         augmented  hypergraphs
on  $S$  and  $S'$  respectively.   
    The  {\it  join}  of  $\mathcal{H} $  and  
$\mathcal{H}' $  
 is  an  augmented   hypergraph  on  $S\sqcup  S'$  given  by
 \begin{eqnarray}\label{eq-098bg}
  \mathcal{H}*\mathcal{H}' = \{\sigma\sqcup  \sigma'\mid \sigma\in\mathcal{H}{\rm~and~}\sigma'\in\mathcal{H}'\}  \cup  \mathcal{H}\cup \mathcal{H}'.  
  \end{eqnarray}
  Note   that   $\emptyset\in  \mathcal{H}*\mathcal{H}'$  iff  $\emptyset\in\mathcal{H}$  or 
  $\emptyset\in \mathcal{H}'$.  
%  consisting  of  the  hyperedges  
%%  $\sigma\sqcup  \sigma'$  where  
%  $\sigma\in  \mathcal{H}$  and  $\sigma'\in \mathcal{H}'$,  $\sigma\mathcal{H}$  and 
%  $\sigma'\in\mathcal{H}'$.  
  %given by  
% \begin{eqnarray*}
% \mathcal{H}   *\mathcal{H}'   = \mathcal{H} \sqcup \mathcal{H}'  \sqcup \{\sigma\sqcup \sigma'\mid  \sigma\in\mathcal{H} ,  \sigma'\in\mathcal{H}' \}.  
% \end{eqnarray*} 
  In  particular,  
  (1)  
   for  any             augmented  simplicial  complexes  $\mathcal{K} $     and  $\mathcal{K}' $     on  $S$  and  $S'$  respectively,  
   $\mathcal{K} *\mathcal{K}' $    is  an  augmented   simplicial  complex  on  $S\sqcup  S'$;  
  (2)  
    for  any    augmented  independence  hypergraphs     $\mathcal{L} $   and  $\mathcal{L}' $     on  $S$  and  $S'$  respectively,  
     $\mathcal{L} *\mathcal{L}' $     is  an  augmented   independence hypergraph 
  on  $S\sqcup  S'$.

 Let    ${\bf   H} (S)$,  ${\bf  K} (S)$   and  ${\bf  L} (S)$      be  the  set  
   of  augmented  hypergraphs on  $S$,  the  set  of  augmented  simplicial  complexes  on  $S$  
   and  the  set   of   augmented  independence  hypergraphs  on  $S$  respectively.   
   Then  ${\bf  K}(S),  {\bf  L}(S)\subset   {\bf  H}(S)$  and  ${\bf  K}(S)  \cap  {\bf  L}(S)= \{[2^S], [2^S]\setminus\{\emptyset\}\}$.   
  We  have  maps    
  \begin{eqnarray*}
   & \Delta,\delta:   ~~~   {\bf   H} (S)\longrightarrow  {\bf   K} (S),\\ 
    &\bar\Delta,\bar\delta:  ~~~  {\bf   H} (S)\longrightarrow  {\bf   L} (S),\\
    &  \cap,\cup:  ~~~ {\bf   H} (S)\times   {\bf   H} (S) \longrightarrow       {\bf   H} (S). %,  ~~~  {\bf   K} (S)\times   {\bf   K} (S) \longrightarrow       {\bf   K} (S),~~~ {\bf   L} (S)\times   {\bf   L} (S) \longrightarrow       {\bf   L} (S). 
    \end{eqnarray*}
    The  restrictions  of  $\Delta$  and   $\delta$   to   ${\bf   K} (S)$  
    are    the  identity  map  while  the  restrictions  of  $\bar\Delta$  and  $\bar\delta$
     to  ${\bf   L} (S)$  are   the  identity  map.  
     Both  $ \cap$  and    $\cup$  send  
     ${\bf   K} (S)\times   {\bf   K} (S)$  to       $ {\bf   K} (S)$  and  send   ${\bf   L} (S)\times   {\bf   L} (S) $  to  $ {\bf   L} (S)$.  Moreover,  for  any  $S$  and  $S'$  disjoint,  we  have  a  map  
    \begin{eqnarray*}
  *:  ~~~  {\bf   H} (S)\times   {\bf   H} (S') \longrightarrow       {\bf   H} (S\sqcup  S')
    \end{eqnarray*}
    which  sends   
       $ {\bf   K} (S)\times   {\bf   K} (S')$  to       ${\bf   K} (S\sqcup  S') $  and  sends   $ {\bf   L} (S)\times   {\bf   L} (S')$  to      $  {\bf   L} (S\sqcup  S')$.  
    Furthermore,  if  
    $S$    is  finite,  then  we  have  maps 
      \begin{eqnarray*}
    &  \gamma_S, \Gamma_S: ~~~   {\bf   H} (S)\longrightarrow {\bf   H} (S)
    \end{eqnarray*}
    such  that  $\gamma_S^2=\Gamma_S^2={\rm  id}$  and  
    both  $\gamma_S$  and  $\Gamma_S$   send    $ {\bf   K} (S)$  to  ${\bf   L} (S)$  and   send   ${\bf   L} (S)$  to
    $ {\bf   K} (S) $.

\begin{theorem}\label{thh=5.1}
For  any  $T\subseteq  S$,   the  trace  of  augmented  hypergraphs  is  a  map  $(-)\mid_T:  {\bf   H}(S) \longrightarrow  {\bf   H}(T)$       such  that  
\begin{enumerate}[(1)]
\item
$(-)\mid_T$  sends     ${\bf   K}  (S)$  to  ${\bf  K} (T)$    and     sends  ${\bf  L}  (S)$  to  ${\bf  L} (T)$;   
\item
$\Delta\circ  (-)\mid_T  = (-)\mid_T  \circ \Delta$   and  $\bar\Delta\circ  (-)\mid_T  = (-)\mid_T  \circ \bar\Delta$; 
\item   
$(-)\mid_T\circ  \cup =\cup \circ  ((-)\mid_T,   (-)\mid_T) $;  
\item
  for   augmented   hypergraph  pairs  $(\mathcal{H},\mathcal{H}')$  such  
that  $\sigma\cap\sigma'\in \mathcal{H}\cap\mathcal{H}'$   for  any  $\sigma\in \mathcal{H}$  and  any $\sigma'\in\mathcal{H}'$,  
   $(-)\mid_T\circ  \cap =\cap \circ  ((-)\mid_T,   (-)\mid_T) $;  
 \item
$(-)\mid_T\circ\Gamma_S=\Gamma_T\circ  (-)\mid_T  $.   
\end{enumerate}
Moreover,  Let  $S$  and  $S'$  be  disjoint  sets.   
For  any  $T\subseteq  S$  and  any  $T'\subseteq  S'$,  
  \begin{enumerate}[(6)]
  \item
  $
(-)\mid_{T\sqcup  T'}\circ  *=   *\circ  ((-)\mid_T,  (-)\mid_{T'}).  
$
\end{enumerate}
\end{theorem}

\begin{proof}
 (1)   To  prove  $(-)\mid_T$  sends      ${\bf   K}  (S)$  to  ${\bf  K} (T)$, 
 it  suffices to  prove 
 that  for  any  augmented   simplicial  complex  $\mathcal{K}$  on  $S$,  $\mathcal{K}\mid _T$  is  a    augmented   simplicial   complex     on  $T$.  
 Let  $\mathcal{K}$  be  an  augmented  simplicial  complex   on  $S$.  
 Let  $\tau\in  \mathcal{K}\mid _T$.  
 Without  loss  of  generality,  suppose  $\tau\neq\emptyset$.  
 Then  there  exists  $\sigma\in\mathcal{K}$  such  that  
 $\tau =\sigma\cap  T$.     
%Thus   $\tau\in \mathcal{K}$.  
 %Hence  $ \mathcal{K}\mid _T\subseteq  \mathcal{K}$.   
   Let  $\eta\subseteq  \tau$  be  a  nomempty  subset  of  $\tau$.  
Then  $\eta=  \sigma'\cap  T$  where  $\sigma'=\eta\cup (\sigma\setminus  \tau)$.  
Since  $\sigma'\subseteq \sigma$,  we  have   $\sigma'\in \mathcal{K}$.   
Thus  $\eta\in  \mathcal{K}\mid _T$.  
Therefore,   $\mathcal{K}\mid _T$  is  an  augmented    simplicial   complex  on  $T$.  % sub-complex

 To  prove  $(-)\mid_T$  sends      ${\bf   L}  (S)$  to  ${\bf  L} (T)$, 
 it  suffices to  prove 
 that   for  any   augmented   independence  hypergraph  $\mathcal{L}$  on  $S$,  $\mathcal{L}\mid _T$  is  an   augmented   independence   hypergraph      on  $T$.  
 Let  $\mathcal{L}$  be  an  augmented  independence  hypergraph  on  $S$.  
 Let  $\tau\in  \mathcal{L}\mid _T$.  
There  exists  $\sigma\in\mathcal{L}$  such  that  
$\tau =\sigma\cap  T$.    
Let  $\tau\subseteq \eta\subseteq  T$.  
Then  $\eta=  \sigma'\cap  T$  where  $\sigma'=\eta\cup (\sigma\setminus  \tau)$.  
Since  $\sigma \subseteq \sigma'$,   we  have  $\sigma'\in \mathcal{L}$.   
Thus  $\eta\in  \mathcal{L}\mid _T$.  
Therefore,   $\mathcal{L}\mid _T$  is  an  augmented  independence  hypergraph  on  $T$.

 (2)   Let  $\mathcal{H}$  be  an   augmented   hypergraph  on  $S$.  
  Let   $\sigma_1\in  \Delta(\mathcal{H}\mid  _T) $.   
    There  exists   $\tau _1\in  \mathcal{H}\mid  _T$  such  that  $\sigma_1\subseteq\tau_1$.   
  There  exists  $\delta_1\in \mathcal{H}$  such  that  $\tau _1=\delta_1\cap  T$.    
  It  follows  that   $\sigma_1\subseteq \delta_1\cap  T$,  which  implies   
    $\sigma_1\in  (\Delta  \mathcal{H})\mid  _T$.  
  Hence   
  $ \Delta(\mathcal{H}\mid  _T)\subseteq  (\Delta  \mathcal{H})\mid  _T$.  
   Let    $\sigma_2 \in (\Delta  \mathcal{H})\mid  _T$.  There  exists  $\tau_2\in  \Delta\mathcal{H}$
  such  that  $\sigma_2= \tau_2\cap  T$.  
  There  exists  $\delta_2\in\mathcal{H}$  such  that  $\tau_2\subseteq\delta_2$.  
      Thus   $\sigma_2\subseteq  \delta_2\cap   T$,  which  implies  $\sigma_2\in  \Delta(\mathcal{H}\mid  _T)$. 
  Hence  $  (\Delta  \mathcal{H})\mid  _T \subseteq  \Delta(\mathcal{H}\mid  _T)$.  
 Therefore,  
 % \begin{eqnarray}\label{eq-6.cc1}
$\Delta(\mathcal{H}\mid  _T)   =  (\Delta  \mathcal{H})\mid  _T$.   
 % \label{eq-6.cc2}
 %\end{eqnarray}  
We  obtain  $\Delta\circ  (-)\mid_T  = (-)\mid_T  \circ \Delta$.

    Let  $\sigma_3\in  \bar \Delta(\mathcal{H}\mid  _T)$.  
   There  exists  $\tau _3\in  \mathcal{H}\mid  _T$  such  that  $\sigma_3\supseteq\tau_3$.     
    There  exists  $\delta_3\in \mathcal{H}$  such  that   $\tau _3=\delta_3\cap  T$.  
 It  follows  that   $\sigma_3\supseteq \delta_3\cap  T$,  which  implies   
    $\sigma_3\in  (\bar\Delta  \mathcal{H})\mid  _T$.  
Hence  $\bar \Delta(\mathcal{H}\mid  _T)\subseteq  (\bar\Delta  \mathcal{H})\mid  _T$.  
 Let    $\sigma_4 \in (\bar\Delta  \mathcal{H})\mid  _T$.  There  exists  $\tau_4\in \bar \Delta \mathcal{H}$
  such  that  $\sigma_4= \tau_4\cap  T$.    
  There  exists  $\delta_4\in\mathcal{H}$  such  that  $\tau_4\supseteq\delta_4$.  
      Thus   $\sigma_4\supseteq  \delta_4\cap   T$,  which  implies
   $\sigma_4\in  \bar\Delta(\mathcal{H}\mid  _T)$. 
  Hence  $  (\bar\Delta  \mathcal{H})\mid  _T \subseteq \bar \Delta(\mathcal{H}\mid  _T)$.  
  Therefore, 
  %\begin{eqnarray}
 $\bar \Delta(\mathcal{H}\mid  _T)   =  (\bar\Delta  \mathcal{H})\mid  _T$.    
  %\label{eq-6.cc2}
  %\end{eqnarray}  
 We  obtain  $\bar\Delta\circ  (-)\mid_T  = (-)\mid_T  \circ \bar\Delta$.

 (3)    Let  $\mathcal{H}$  and  $\mathcal{H}'$   be  augmented  hypergraphs  on  $S$. 
   Let  $\tau _1\in  (\mathcal{H}\cup\mathcal{H}')\mid _T$.   
Then  there  exists  $\epsilon_1\in \mathcal{H}\cup\mathcal{H}'$  such  that  
$\tau_1  = \epsilon_1\cap  T$.  
It  follows  that  $\tau_1 \in   (\mathcal{H}\mid_T) \cup  (\mathcal{H}'\mid_T) $.   
Conversely,  let  
$\delta_1\in    (\mathcal{H}\mid_T) \cup  (\mathcal{H}'\mid_T) $.  
Then  either  there  exists  $\sigma_1\in \mathcal{H}$  such that  
$\tau_1=\sigma_1\cap  T$  or   there  exists  $\sigma'_1\in \mathcal{H}'$  such that  
$\tau_1=\sigma'_1\cap  T$.   Thus there  exists  $\eta_1\in \mathcal{H}\cup\mathcal{H}'$  such  that  
$\tau  = \eta_1\cap  T$.  
It  follows  that  $\delta_1\in   (\mathcal{H}\cup\mathcal{H}')\mid _T$.  
We  obtain  $ (\mathcal{H}\cup\mathcal{H}')\mid _T =  (\mathcal{H}\mid_T) \cup  (\mathcal{H}'\mid_T)  
$.

 (4)   Let  $\mathcal{H}$  and  $\mathcal{H}'$   be  augmented  hypergraphs  on  $S$.  
 Let  $\tau_2 \in  (\mathcal{H}\cap\mathcal{H}')\mid _T$.   
Then  there  exists  $\epsilon_2\in \mathcal{H}\cap\mathcal{H}'$  such  that  
$\tau_2  = \epsilon_2\cap  T$.  
Hence  $\tau_2\in  (\mathcal{H}\mid_T )\cap  (\mathcal{H}'\mid_T)$.  
We  obtain  $
(\mathcal{H}\cap\mathcal{H}')\mid _T \subseteq    (\mathcal{H}\mid_T )\cap  (\mathcal{H}'\mid_T)    
$. 
In  addition, suppose  $\sigma_2\cap\sigma'_2\in  \mathcal{H}\cap\mathcal{H}'$ 
 for  any  $\sigma_2\in \mathcal{H}$  and  any  $\sigma'_2\in\mathcal{H}'$.    
 Let  $\delta_2 \in   (\mathcal{H}\mid_T )\cap  (\mathcal{H}'\mid_T)$.  
 Then  there  exist   $\eta_2\in \mathcal{H}$  and  $\eta'_2\in \mathcal{H}'$  such  that  
  $\delta _2=\eta_2\cap  T=\eta'_2\cap   T$.  
  Consequently,  
  $\delta_2= (\eta_2\cap\eta'_2) \cap  T$.  
  By  our  assumption,  $\eta_2\cap\eta'_2\in  \mathcal{H}\cap\mathcal{H}'$.  
  Hence  $\delta_2\in  (\mathcal{H}\cap\mathcal{H}')\mid _T$.  
  We  obtain  $
(\mathcal{H}\cap\mathcal{H}')\mid _T =    (\mathcal{H}\mid_T )\cap  (\mathcal{H}'\mid_T)    
$.

  (5)   Suppose  $S$  is  finite.   
  For  any     augmented  hypergraph   $\mathcal{H}$  on  $S$,  
 $
(\Gamma_S\mathcal{H} )\mid_T  = \{(S\setminus \sigma)\cap  T\mid  \sigma\in   \mathcal{H} \} 
 =  \{ T\setminus( \sigma \cap  T)\mid  \sigma\in   \mathcal{H} \} 
 =   \Gamma_T(\mathcal{H}\mid_T)$.    
  Therefore,  $(-)\mid_T\circ\Gamma_S=\Gamma_T\circ  (-)\mid_T  $.

  (6)  Let  $S$  and   $S'$  be  disjoint.  
  Let  $\mathcal{H} $  and     $\mathcal{H}' $    be      augmented    hypergraphs  on  $S$  and   $S'$
 respectively.  
Let  $T\subseteq  S$  and    $T'\subseteq  S'$.
By     (\ref{eq-098bg}),  
\begin{eqnarray*}
(\mathcal{H}*\mathcal{H}')\mid_{T\sqcup  T'} &=&  (\{\sigma\sqcup\sigma'\mid  \sigma\in  \mathcal{H} {\rm~and~}   \sigma'\in  \mathcal{H}' \}   \cup \mathcal{H} \cup \mathcal{H}')\mid_{T\sqcup  T'}\\
&=&(\{\sigma\sqcup\sigma'\mid  \sigma\in  \mathcal{H} {\rm~and~}   \sigma'\in  \mathcal{H}' \} \mid _{T\sqcup  T'})\cup  (\mathcal{H}\mid_T)  \cup (\mathcal{H}'\mid_{T'})\\
&=&  \{(\sigma\cap  T)\sqcup(\sigma'\cap  T')\mid  \sigma\in  \mathcal{H} {\rm~and~}  \sigma'\in  \mathcal{H}' \} \cup  (\mathcal{H}\mid_T)  \cup (\mathcal{H}'\mid_{T'})
\\
&=&  \{\tau\sqcup  \tau\mid   \tau\in  \mathcal{H}\mid_T {\rm~and~} \tau'\in  \mathcal{H}'\mid_{T'}\} \cup  (\mathcal{H}\mid_T)  \cup (\mathcal{H}'\mid_{T'})
\\
 &=&(\mathcal{H}\mid  _T)  *(\mathcal{H}'\mid_{T'}), 
\end{eqnarray*}
where  $ \tau=\sigma\cap T $  and  $ \tau'=\sigma'\cap T'$.  
Therefore,  $(-)\mid_{T\sqcup  T'}\circ  *=   *\circ  ((-)\mid_T,  (-)\mid_{T'})$.    
\end{proof}

Let  $\mathcal{H}$  and     $\mathcal{H}'$  be    augmented  hypergraphs  on  $S$  
 and   $S'$  respectively.  
A  {\it  morphism}  $\varphi:  \mathcal{H}\longrightarrow  \mathcal{H}'$  
is  a  map  $\varphi_0:  S\longrightarrow  S'$  such  that  for  any  $\sigma\in \mathcal{H}$,  
its  image    
%\begin{eqnarray}\label{eq-aa.1}
  $\varphi(\sigma)= \{\varphi_0(s)\mid  s\in\sigma\}$ 
%\end{eqnarray}
  is  a  hyperedge  in  $\mathcal{H}'$.  
  Such  a  morphism  must  send   $\emptyset \in\mathcal{H}$  to  $\emptyset\in \mathcal{H}'$.   
In  particular,  let  $\mathcal{K}$  and  $\mathcal{K}'$  
be  augmented  simplicial  complexes  on  $S$  and  $S'$  respectively.   
A  morphism  $\varphi:  \mathcal{K}\longrightarrow \mathcal{K}'$  is  a  {\it simplicial  map}.  
Let  $\mathcal{L}$  and  $\mathcal{L}'$  
be  augmented  independence  hypergraphs   on  $S$  and  $S'$  respectively.
 A  morphism $\varphi:  \mathcal{L}\longrightarrow \mathcal{L}'$  is  a  {\it  morphism  of  augmented  independence  hypergraphs}.       
%take  $\mathcal{H}=\mathcal{K}(n_1,n_2,\cdots)$  and    $\mathcal{H}'=\mathcal{K}'(n'_1,n'_2,\ldots)$.   

  Let  $\varphi:  \mathcal{H}\longrightarrow \mathcal{H}'$  be   a  morphism  between  augmented  
  hypergraphs  induced  by  a  map  $\varphi_0:  S\longrightarrow  S'$.        
 Then  the  diagram   commutes  
  \begin{eqnarray}\label{eq-vmbg1}
  \xymatrix{
 \delta \mathcal{H}\ar[r]^-{\delta\varphi} \ar[d] &\delta\mathcal{H}' \ar[d]\\
  \mathcal{H}\ar[r]^-{\varphi} \ar[d] &\mathcal{H}' \ar[d] \\
    \Delta\mathcal{H} \ar[r]^-{\Delta\varphi}    & \Delta\mathcal{H}'     
  }
  \end{eqnarray}
where  both  $\delta\varphi$  and  $\Delta\varphi$  are       simplicial  maps  induced  by  
$\varphi_0$  and 
 all  the  vertical  maps  are  canonical  inclusions.    In  addition,  if    $\varphi_0:  S\longrightarrow  S'$  is  bijective,  then   
  the  diagram  commutes   
  \begin{eqnarray}\label{eq-vmbg2}
  \xymatrix{
 \bar\delta \mathcal{H}\ar[r]^-{\bar\delta\varphi} \ar[d] &\bar\delta\mathcal{H}' \ar[d]\\
  \mathcal{H}\ar[r]^-{\varphi} \ar[d] &\mathcal{H}' \ar[d] \\
    \bar\Delta\mathcal{H} \ar[r]^-{\bar\Delta\varphi}    & \bar\Delta\mathcal{H}'     
  }
  \end{eqnarray}
 where  both  $\bar\delta\varphi$  and  $\bar\Delta\varphi$  are    morphisms  of  augmented  independence  hypergraphs  induced  by  
$\varphi_0$     and 
 all  the  vertical  maps  are  canonical  inclusions.  
   Let  $T\subseteq  S$  and  let  $T'=\varphi_0(T) \subseteq  S'$.  
  Then  we  have  an  induced  morphism  
  $\varphi\mid_T:  \mathcal{H}\mid_T\longrightarrow  \mathcal{H}'\mid _{T'}$
  sending  $\sigma \cap  T$   to  $\varphi(\sigma)\cap  T'$  for  any  $\sigma\in \mathcal{H}$.

 \begin{corollary}\label{le-99mb1}
 For  any  morphism  $\varphi:  \mathcal{H}\longrightarrow  \mathcal{H}'$  between  augmented  hypergraphs  induced  by  $\varphi_0:  S\longrightarrow  S'$,   
 \begin{eqnarray}\label{eq-6.dd1}
 \Delta(\varphi\mid_T)=  (\Delta\varphi)\mid_T.
  \end{eqnarray}
 Moreover,  if   $\varphi_0:  S\longrightarrow  S'$  is  bijective,  then   
\begin{eqnarray} \label{eq-6.dd2}
  \bar\Delta(\varphi\mid_T)=  (\bar\Delta\varphi)\mid_T. 
  \end{eqnarray}
 \end{corollary}
 
 \begin{proof}
 Let  $\varphi:  \mathcal{H}\longrightarrow  \mathcal{H}'$  be  a  morphism.  
 We  have  an  induced   morphism   $\varphi\mid_T:  \mathcal{H}\mid_T\longrightarrow  \mathcal{H}'\mid _{T'}$.  By  the  first identity  in Theorem~\ref{thh=5.1}~(2)    and  (\ref{eq-vmbg1}),  the   diagram  commutes 
 \begin{eqnarray*}
 \xymatrix{
 \mathcal{H}\mid_T\ar[r]^-{\varphi\mid_T} \ar[d] &\mathcal{H}' \mid_{T'}\ar[d] \\
 \Delta( \mathcal{H}\mid_T )  \ar[r]^-{\Delta(\varphi\mid_T)} \ar@{=}[d] 
   & \Delta(\mathcal{H}' \mid_{T'})\ar@{=}[d]\\
    (\Delta\mathcal{H})\mid_T \ar[r]^-{(\Delta\varphi)\mid_T}    & (\Delta\mathcal{H}')\mid_{T'}.  
 }
 \end{eqnarray*}
 We  obtain  (\ref{eq-6.dd1}).   Suppose  in  additon  $\varphi_0:  S\longrightarrow  S'$  is  bijective.  
 Then   by   the  second  identity  in Theorem~\ref{thh=5.1}~(2)   and  (\ref{eq-vmbg2}),  the   diagram  commutes 
 \begin{eqnarray*}
 \xymatrix{
 \mathcal{H}\mid_T\ar[r]^-{\varphi\mid_T} \ar[d] &\mathcal{H}' \mid_{T'}\ar[d] \\
 \bar\Delta( \mathcal{H}\mid_T )  \ar[r]^-{\bar\Delta(\varphi\mid_T)} \ar@{=}[d] 
   & \bar\Delta(\mathcal{H}' \mid_{T'})\ar@{=}[d]\\
    (\bar\Delta\mathcal{H})\mid_T \ar[r]^-{(\bar\Delta\varphi)\mid_T}    & (\bar\Delta\mathcal{H}')\mid_{T'}.  
 }
 \end{eqnarray*}
 We  obtain  (\ref{eq-6.dd2}). 
 \end{proof}

\begin{example}
Let  $S=\{s_0,s_1,s_2\}$  and  $T= \{s_0,s_1\}$.  
\begin{enumerate}[(1)]
\item
Let  $\mathcal{H}=  \{\emptyset, \{s_0\},  \{s_0,s_1\},  \{s_0,s_2\},    \{s_1,s_2\},   \{s_0,s_1,s_2\}\}$.  
Then    
\begin{eqnarray*}
&\Delta\mathcal{H}=[2^S],  ~~~~~~
 \delta\mathcal{H}= \{\emptyset,  \{s_0\}\},  
~~~~~~\bar\Delta\mathcal{H}  =[ 2^S],\\
& \bar\delta\mathcal{H} =  \{\{s_0\},  \{s_0,s_1\},  \{s_0,s_2\},    \{s_1,s_2\},   \{s_0,s_1,s_2\}\},\\
&\gamma_S\mathcal{H}  =\{\{s_1\},  \{s_2\}\}\\   
& \Gamma_S\mathcal{H}  =  \{ \{s_0,s_1,s_2\},  \{s_1,s_2\},  \{s_2\}, \{s_1\},  \{s_0\}, \emptyset\}. 
\end{eqnarray*} 
Moreover,  
\begin{eqnarray*}
&\mathcal{H}\mid_T= ( \Delta\mathcal{H})\mid  _T = (\bar\Delta\mathcal{H}) \mid_T =(\Gamma_S \mathcal{H})\mid_T = \{\emptyset, \{s_0\},   \{s_1\}, \{s_0,s_1\}  \},\\
&   (\delta\mathcal{H})\mid_T=  \{\emptyset,  \{s_0\}\},~~~~~~
      (\bar\delta\mathcal{H})\mid_T=  \{ \{s_0\},\{s_1\},  \{s_0,s_1\}\},~~~~~~
  (\gamma_S\mathcal{H})\mid_T  = \{\{s_1\}\}.  
 \end{eqnarray*}
 
\item
Let  $\mathcal{H}'=   \{\emptyset, \{s_0\}, \{s_1\}, \{s_0,s_1\},  \{s_0,s_2\},    \{s_1,s_2\},   \{s_0,s_1,s_2\}\}$.  Then 
\begin{eqnarray*}
&\Delta\mathcal{H}'=[2^S],   
~~~~~~\delta\mathcal{H}'= \{\emptyset,  \{s_0\}, \{s_1\},  \{s_0,s_1\}\},  
~~~~~~\bar\Delta\mathcal{H}'  =[ 2^S],  \\
&\bar\delta\mathcal{H} '=  \{\{s_0\},  \{s_1\},  \{s_0,s_1\},  \{s_0,s_2\},    \{s_1,s_2\},   \{s_0,s_1,s_2\}\},\\
&\gamma_S\mathcal{H}' =  \{\emptyset, \{s_2\}\}, \\ 
& \Gamma_S\mathcal{H}'=\{\{s_0,s_1,s_2\}, \{s_1,s_2\},  \{s_0,s_2\}, \{s_2\},  \{s_1\}, \{s_0\},  \emptyset\}.   
\end{eqnarray*}  
Moreover,  
\begin{eqnarray*}
&\mathcal{H}'\mid_T =(\Delta\mathcal{H}')\mid_T=(\delta\mathcal{H}')\mid_T=(\bar\Delta\mathcal{H}')\mid_T=(\Gamma_S \mathcal{H}')\mid_T = \{\emptyset, \{s_0\},   \{s_1\}, \{s_0,s_1\}  \},\\
&(\bar\delta\mathcal{H} ')\mid_T =    \{ \{s_0\},   \{s_1\}, \{s_0,s_1\}  \},~~~~~~
 (\gamma_S\mathcal{H}' )=\{\emptyset\}. 
\end{eqnarray*} 

\item
Let  $\mathcal{H}$  be  given in  (1)  and  let  $\mathcal{H}'$  be  given  in  (2).  
Let  $\varphi:  \mathcal{H}\longrightarrow  \mathcal{H}'$  be  the  canonical  inclusion  of  augmented  hypergraphs  induced  by  the  identity  map  on  $S$.  Then  $\varphi$  induces  simplicial  maps  
\begin{eqnarray*}
\Delta\varphi:  ~~~  \Delta\mathcal{H}\longrightarrow  \Delta\mathcal{H}',~~~~~~  
\delta\varphi: ~~~ \delta\mathcal{H}\longrightarrow  \delta\mathcal{H}'
\end{eqnarray*}
and  morphisms  of   augmented  independence  hypergraphs 
\begin{eqnarray*}
\bar\Delta\varphi:  ~~~  \bar\Delta\mathcal{H}\longrightarrow  \bar\Delta\mathcal{H}',~~~~~~  
\bar\delta\varphi: ~~~ \bar\delta\mathcal{H}\longrightarrow  \bar\delta\mathcal{H}'   
\end{eqnarray*}
such  that  $ \Delta(\varphi\mid_T)=  (\Delta\varphi)\mid_T$  and   $\bar\Delta(\varphi\mid_T)=  (\bar\Delta\varphi)\mid_T$.  
\end{enumerate}

\end{example}

\section{Invariant  traces  of  hypergraphs}\label{s---6}

Let  $S$  be  a  set  of  vertices  with  a  total  order.  
  For  each  hyperedge  $\{s_0,s_1,\ldots, s_n\}$  on  $S$, 
 we  choose  the  unique     representative  $s_0s_1\ldots  s_n$ such  that  $s_0\prec  s_1\prec 
\cdots \prec  s_n$.  
Then  the  set  $[2^S]$  can  be   identified  with  the  set  of   all  simplicial  acyclic  elements  in    $\langle  S\rangle$  together  with  $\emptyset$. 
Consequently,   the       $R$-module  
     $R_n( [2^S])$  can  be  identified  with  $\mathcal{F}\langle  S\rangle_n$  and  
     the  $R$-module    $R_{-1}( [2^S])$  can  be  identified  with 
     $  R$.

For  any  augmented  hypergraph  $\mathcal{H}$  on  $S$ and  any  $n\geq  0$,  
let   $R_n(\mathcal{H})$  be  the  free  $R$-module  generated  by  
 all  the  $n$-hyperedges  in  $\mathcal{H}$.   
 Let  $R_{-1}(\mathcal{H})=0$   if  $\emptyset\notin \mathcal{H}$  and  let 
 $R_{-1}(\mathcal{H}) =R $  if  $\emptyset\in\mathcal{H}$.  
  Let  $  R_*(\mathcal{H})= \oplus_{n\geq   0}  R_n(\mathcal{H})$  and  let  
$\tilde  R_*(\mathcal{H})= \oplus_{n\geq   -1}  R_n(\mathcal{H})$.  
Since  $\mathcal{H}$  is  a  subset  of  $[2^S]$,  
  $\tilde  R_*(\mathcal{H})$  is  a  graded  sub-$R$-module  of  
   $\tilde    R_*( [2^S])$.

     \begin{definition}
     Let  $\mathcal{H}$  be  an  augmented  hypergraph  on  $S$.   Let  $s\in  S$.  
     \begin{enumerate}[(1)]
     \item
      We  say  that   $    R_*(\mathcal{H})$  is   {\it  $\partial/\partial  s$-invariant}   if  
 $\partial_i/\partial  s:    R_{n+1}(\mathcal{H})\longrightarrow    R_{n}(\mathcal{H})$  is  well-defined  for   any      $n\in \mathbb{N}$  and  any  
$0\leq   i\leq   n+1$;
\item  
 We  say  that      $     R_*(\mathcal{H})$  is   {\it  $d  s$-invariant}   if  
$d_i s:  R_n(\mathcal{H})\longrightarrow  R_{n+1}(\mathcal{H})$  is  well-defined  for   any   $n\in \mathbb{N}$  and  any  
$0\leq   i\leq   n+1$.    
 \end{enumerate}
 \end{definition}
 
   \begin{lemma}\label{le-6.1x}
     For    any   $s\in  S$   and    any  augmented  hypergraph  $\mathcal{H}$  on  $S$,  
\begin{enumerate}[(1)]
\item
$R_*(\mathcal{H})$  is   $\partial/\partial  s$-invariant  iff   for  any  $\sigma\in \mathcal{H}$  
such  that  $s\in \sigma$,    if  $\sigma\setminus \{s\}\neq \emptyset$  then  $(\sigma\setminus \{s\})\in \mathcal{H}$;  
\item
  $R_*(\mathcal{H})$  is   $d s$-invariant  iff    for  any  $\sigma\in \mathcal{H}$  
such  that  $s\notin \sigma$  and  $\sigma\neq\emptyset$,     it  holds    $(\sigma\cup \{s\})\in \mathcal{H}$.    
\end{enumerate}
     \end{lemma}
     
     \begin{proof}
   (1)   $R_*(\mathcal{H})$  is   $\partial/\partial  s$-invariant  
    iff 
     for  \ any  $n\in \mathbb{N}$,  any   $0\leq  i\leq   n+1$
  and   any   $ \{s_0,s_1,\ldots, s_{n+1}\}\in  \mathcal{H}$  such  that  $s_i=s$,  
it  holds   $  \{s_0, \ldots, \widehat{s_i}, \ldots, s_{n+1}\}\in  \mathcal{H}$.   
  Write  $\sigma=  \{s_0,s_1,\ldots, s_{n+1}\}$.  We  obtain  (1).

  (2)  $R_*(\mathcal{H})$  is   $d  s$-invariant  
    iff   for  any      $n\in \mathbb{N}$,   any     $0\leq  i\leq   n+1$
     and    any   $ \{s_0,s_1,\ldots, s_n\}\in  \mathcal{H}$  such  that  $s_{i-1}\prec  s\prec  s_i$,  
 we  have    $  \{s_0, \ldots,  s_{i-1},  s,  s_i,  \ldots, s_n\}  \in  \mathcal{H}$.  
    Write  $\sigma=  \{s_0,s_1,\ldots, s_n\}$.  We  obtain  (2).  
     \end{proof}

 \begin{lemma}\label{le-6.2bva}
For  any  augmented   hypergraph  $\mathcal{H}$  on  $S$,  
\begin{enumerate}[(1)]
\item
$\mathcal{H}$  is  an  augmented   simplicial  complex     iff  $R_*(\mathcal{H})$  is   $\partial/\partial  s$-invariant for   any   $s\in   S$;   
\item
$\mathcal{H}$  is  an  augmented  independence  hypergraph      iff  $R_*(\mathcal{H})$  is   $d  s$-invariant for   any   $s\in   S$.    
\end{enumerate}
 \end{lemma}
 
 \begin{proof}
 (1)   $\mathcal{H}$  is  a  simplicial  complex   iff
    for  any  $\sigma\in \mathcal{H}$  
 and  any   $s\in \sigma$,    if  $\sigma\setminus \{s\}\neq \emptyset$  then  $(\sigma\setminus \{s\})\in \mathcal{H}$.  
 By  Lemma~\ref{le-6.1x}~(1),   this  holds  iff   $R_*(\mathcal{H})$  is   $\partial/\partial  s$-invariant for   any  
 $s\in  \bigcup_{\sigma\in \mathcal{H},  \dim\sigma\geq  1}  \sigma$,  
   iff    $R_*(\mathcal{H})$  is   $\partial/\partial  s$-invariant for   any  
  $s\in   S$.

 (2)    $\mathcal{H}$  is  an   independence  hypergraph   iff   
  for  any    $s\in  S$  and   any  $\sigma\in \mathcal{H}$  
such  that  $s\notin \sigma$   and   $\sigma\neq\emptyset$,     it  holds    $(\sigma\cup \{s\})\in \mathcal{H}$.   
 By  Lemma~\ref{le-6.1x}~(2),   this  holds  iff   $R_*(\mathcal{H})$  is   $d  s$-invariant for   any
 $s\in  \bigcup_{\sigma\in \mathcal{H}, \sigma\neq\emptyset}\Gamma_S(\sigma)$,  
 iff  
  $R_*(\mathcal{H})$  is   $d  s$-invariant for   any  
    $s\in   S$. 
 \end{proof}

 Let  $S( \mathcal{H},\partial)$   be   the  subset  of  $S$  consisting  of  all  $s\in  S$  such  that  
    $ R_*(\mathcal{H})$  is   $\partial/\partial  s$-invariant.     Let  $S(\mathcal{H},d)$   be   the  subset  of  $S$  consisting  of  all  $s\in  S$  such  that  
    $ R_*(\mathcal{H})$  is   $d  s$-invariant.   
    
    \begin{theorem}\label{pr-5.5aaa}
\begin{enumerate}[(1)]
\item
    For  any  augmented  hypergraph  $\mathcal{H}$  on  $S$,   $\mathcal{H}\mid _{S( \mathcal{H},\partial)}$  is  an  augmented    simplicial  complex  on   $S( \mathcal{H},\partial)$.  Moreover,   for  any   augmented   simplicial  complex   $\mathcal{K}$    on  $S$,   we  have  $S( \mathcal{K},\partial)=S$;    
\item
    For  any   hypergraph  $\mathcal{H}$  on  $S$,   $\mathcal{H}\mid_{S(\mathcal{H},d)}$  is  an     independence  hypergraph on  $S(\mathcal{H},d)$.   Moreover,     for  any      independence  hypergraph
$\mathcal{L}$   on  $S$,  we  have   $S(\mathcal{L},d)=S$.
  \end{enumerate}
    \end{theorem}

\begin{proof}
(1) 
Let $\mathcal{H}$  be  an  augmented hypergraph  on  $S$.  
  Let   $s\in   S( \mathcal{H},\partial)$.   For  any  $n\in \mathbb{N}$  and  any  $0\leq  i\leq  n$,  
  since   $R_n(\mathcal{H}\mid _{S( \mathcal{H},\partial)})$  is  a  sub-$R$-module  of  
  $R_n(\mathcal{H} )$,  we  have  an   $R$-linear  map   
  \begin{eqnarray*}%\label{eq-0mz-1}
  \frac{\partial_i}{\partial  s}: &&  R_n(\mathcal{H}\mid _{S( \mathcal{H},\partial)})\longrightarrow  R_{n-1}(\mathcal{H}).
  \end{eqnarray*}   
  For  any  $\sigma\in  \mathcal{H}\mid _{S( \mathcal{H},\partial)}$  such  that  $s\in  \sigma$,  
  write  $\sigma=\tau\cap  S( \mathcal{H},\partial)$  for  some  $\tau\in  \mathcal{H}$.  
 Then  
 \begin{eqnarray*}
   \frac{\partial_i}{\partial  s}  (\sigma)  =  \pm   \Big(\frac{\partial_i}{\partial  s}  (\tau)\cap  S( \mathcal{H},\partial)\Big)  \in  R_*(\mathcal{H}\mid _{S( \mathcal{H},\partial)}). 
 \end{eqnarray*}
Therefore, 
 \begin{eqnarray*}%\label{eq-0mz-1}
  \frac{\partial_i}{\partial  s}\Big(  R_n(\mathcal{H}\mid _{S( \mathcal{H},\partial)})\Big)  \subseteq   R_{n-1}(\mathcal{H}\mid _{S( \mathcal{H},\partial)}).
  \end{eqnarray*}   
  By  Lemma~\ref{le-6.2bva}~(1),   we   obtain  (1).  %$\mathcal{H}\mid _{S( \mathcal{H},\partial)}$  
% is  a  simplicial  complex  on   $S$  such that   if  $\mathcal{H}$  is  a  simplicial  complex on  $S$,  then  $\mathcal{H}\mid _{S( \mathcal{H},\partial)}=\mathcal{H}$.   

  (2)  
  Let  $\mathcal{H}$  be  a  hypergraph  on  $S$.  
  Let   $s\in   S( \mathcal{H}, d)$.   For  any  $n\in \mathbb{N}$  and  any  $0\leq  i\leq  n+1$,  
  since   $R_n(\mathcal{H}\mid _{S( \mathcal{H},d)})$  is  a  sub-$R$-module  of  
  $R_n(\mathcal{H} )$,  we  have  an   $R$-linear  map   
  \begin{eqnarray*}%\label{eq-0mz-1}
 d_is: &&  R_n(\mathcal{H}\mid _{S( \mathcal{H},d)})\longrightarrow  R_{n+1}(\mathcal{H}).
  \end{eqnarray*}   
  For  any  $\sigma\in  \mathcal{H}\mid _{S( \mathcal{H},d)}$  such  that  $s\notin  \sigma$,  
  write  $\sigma=\tau\cap  S( \mathcal{H},d)$  for  some  $\tau\in  \mathcal{H}$. 
  If   $s\in \tau$,    then  $s\in  \tau\cap  S( \mathcal{H},d) =\sigma$,  which  contradicts with 
  our assumption.  Thus    $s\notin \tau$.      It  follows  that  
 \begin{eqnarray*}
   d_i s  (\sigma)  =  \pm    ( d_i s  (\tau)\cap  S( \mathcal{H},d) )  \in  R_*(\mathcal{H}\mid _{S( \mathcal{H},d)}). 
 \end{eqnarray*}
Therefore, 
 \begin{eqnarray*}%\label{eq-0mz-1}
  d_i   s  (  R_n(\mathcal{H}\mid _{S( \mathcal{H},d)}) )  \subseteq   R_{n+1}(\mathcal{H}\mid _{S( \mathcal{H},d)}).
  \end{eqnarray*}   
    By  Lemma~\ref{le-6.2bva}~(2),    we   obtain   (2).  %$\mathcal{H}\mid _{S( \mathcal{H},d)}$  
 % is  an  independence  hypergraph  on   $S$.   
      \end{proof}
      
      \begin{remark}
      In  Theorem~\ref{pr-5.5aaa}~(1),   $\emptyset\in \mathcal{H}$  iff  $\emptyset\in   \mathcal{H}\mid_{S(\mathcal{H},\partial)}$.   Note  that  if  $\emptyset\in \mathcal{H}$,  then  
      $\mathcal{H}\mid_{S(\mathcal{H},d)}$  may  not  be  an  augmented  independence  hypergraph.  
      \end{remark}
    
     \begin{corollary}
    \begin{enumerate}[(1)]
    \item
     Let  $\mathcal{H}$  and  $\mathcal{H}'$  be   augmented  hypergraphs  on  $S$  and  $S'$  
     respectively.  
    For  any   surjective   morphism      $\varphi:  \mathcal{H}\longrightarrow  \mathcal{H}'$,   
    we  have  a  simplicial  map  
    \begin{eqnarray}\label{eq-mmmc1}
    \varphi\mid _{S(\mathcal{H},\partial)}:   ~~~\mathcal{H}\mid _{S(\mathcal{H},\partial)}
    \longrightarrow   \mathcal{H}'\mid  _{S'(\mathcal{H}',\partial)}.  
    \end{eqnarray}
    \item
     Let  $\mathcal{H}$  and  $\mathcal{H}'$  be     hypergraphs  on  $S$  and  $S'$  
     respectively.  
    For  any   surjective   morphism      $\varphi:  \mathcal{H}\longrightarrow  \mathcal{H}'$,    if    $\varphi_0:  S\longrightarrow  S'$  is  bijective,  then   
 we  have  a  morphism  of  independence  hypergraphs 
   \begin{eqnarray}\label{eq-mmmc2}
    \varphi\mid _{S(\mathcal{H},d)}:   ~~~\mathcal{H}\mid _{S(\mathcal{H},d)}
    \longrightarrow   \mathcal{H}'\mid  _{S'(\mathcal{H}',d)}.  
    \end{eqnarray}
    \end{enumerate}
    \end{corollary}
    
    \begin{proof}
    (1)  
    Let  $s\in  S(\mathcal{H},\partial)$.  
    Then    $R_*(\mathcal{H})$  is    $\partial/\partial  s$-invariant. 
    With  the  help of  Lemma~\ref{le-6.1x}~(1),     
     $R_*(\varphi(\mathcal{H}))$  is  $\partial/\partial   \varphi_0(s)$-invariant.  
    Since  $\varphi: \mathcal{H}\longrightarrow  \mathcal{H}'$  is  surjective, 
     $R_*( \mathcal{H}')$    is  $\partial/\partial   \varphi_0(s)$-invariant.  
     Thus  $\varphi_0(s)\in  S'(\mathcal{H}',\partial)$.
     It  follows  that  $\varphi_0( S(\mathcal{H},\partial))\subseteq S'(\mathcal{H}',\partial)$.  
     By  Theorem~\ref{pr-5.5aaa}~(1),  
     both  $\mathcal{H}\mid _{S(\mathcal{H},\partial)}$  and 
     $\mathcal{H}'\mid  _{S'(\mathcal{H}',\partial)}$
     are  augmented  simplicial  complexes.  
     With  the  help of  (\ref{eq-6.dd1}),  we  obtain  the  simplicial  map  
     (\ref{eq-mmmc1}).

     (2)  
     Suppose    $\varphi_0:  S\longrightarrow  S'$  is  bijective.  
      Let  $s\in  S(\mathcal{H},d)$.  
    Then    $R_*(\mathcal{H})$  is    $d  s$-invariant.
   With  the  help of  Lemma~\ref{le-6.1x}~(2),     
  $R_*(\varphi(\mathcal{H}))$  is  $d  \varphi_0(s)$-invariant.   
  Since  $\varphi: \mathcal{H}\longrightarrow  \mathcal{H}'$  is  surjective, 
     $R_*( \mathcal{H}')$    is  $d \varphi_0(s)$-invariant.  
       Thus  $\varphi_0(s)\in  S'(\mathcal{H}',d)$.
     It  follows  that  $\varphi_0( S(\mathcal{H},d))\subseteq S'(\mathcal{H}',d)$.  
     By  Theorem~\ref{pr-5.5aaa}~(2),  
     both  $\mathcal{H}\mid _{S(\mathcal{H},d)}$  and 
     $\mathcal{H}'\mid  _{S'(\mathcal{H}',d)}$
     are   independence  hypergraphs.  
     With  the  help of  (\ref{eq-6.dd2}),  we  obtain  the  morphism   
     (\ref{eq-mmmc2}).  
    \end{proof}
    
    \begin{example}
    Let   $S=\{s_0,s_1,s_2 \}$ such  that  $s_0\prec  s_1\prec  s_2$.  
    \begin{enumerate}[(1)]
     \item
    Let   $\mathcal{H}   =\{\{s_1\},  \{s_2\},  \{s_0,  s_1\}, \{s_0,s_2\}\}$. 
    Then  
    \begin{eqnarray*}
  &  \dfrac{\partial}{\partial  s_0}(\{s_1\})= \dfrac{\partial}{\partial  s_0}(\{s_2\}) =\dfrac{\partial}{\partial  s_1}(\{s_1\})=\dfrac{\partial}{\partial  s_1}(\{s_2\})=\dfrac{\partial}{\partial  s_2}(\{s_1\})=\dfrac{\partial}{\partial  s_2}(\{s_2\})=0,\\
&  
  \dfrac{\partial}{\partial  s_1}(\{s_0,s_2\})=\dfrac{\partial}{\partial  s_2}(\{s_0,s_1\})= 0, ~~~
   \dfrac{\partial}{\partial  s_2}(\{s_0,s_2\})=\dfrac{\partial}{\partial  s_1}(\{s_0,s_1\})= -\{s_0\}, \\
  & \dfrac{\partial}{\partial  s_0}(\{s_0, s_1\})=  \{s_1\},  ~~~  \dfrac{\partial}{\partial  s_0}(\{s_0, s_2\})=  \{s_2\}.   
    \end{eqnarray*}
    It  follows  that   $S(\mathcal{H},\partial)=\{s_0\}$  and  consequently 
    $\mathcal{H}\mid_{S(\mathcal{H},\partial)}=  \{\{s_0\}\}$.   
  Moreover,  
  \begin{eqnarray*}
  &    ds_1(\{s_0,  s_1\})=ds_2 (\{s_0,  s_2\})=ds_0(\{s_0,  s_1\})=  ds_0(\{s_0,  s_2\})=0, \\
  & ds_1 (\{s_1\})= ds_2(\{s_2\}) =0,~~~ds_0 (\{s_1\}) = \{s_0,s_1\},~~~
  ds_0(\{s_2\})= \{s_0,s_2\},\\
   &   ds_1(\{s_2\})= \{s_1,s_2\},~~~  ds_2(\{s_1\})=  -\{s_1,s_2\}. 
  \end{eqnarray*}
It  follows  that   $S(\mathcal{H}, d)=\{s_0\}$  and  consequently 
    $\mathcal{H}\mid_{S(\mathcal{H},d)}=  \{\{s_0\}\}$.

    \item
    Let  $\mathcal{H}'=\{\{s_0\}, \{s_1\},  \{s_2\},   \{s_0,s_1\}, \{s_0,s_2\}, \{s_0,s_1,s_2\}\}$.   Then  
    \begin{eqnarray*}
      &  
   % \dfrac{\partial}{\partial  s_1} (\{s_0 ,s_2\})=
   % \dfrac{\partial}{\partial  s_2} (\{s_0 ,s_1\})=0,~~~
         \dfrac{\partial}{\partial  s_1} (\{s_0,s_1 \})=\dfrac{\partial}{\partial  s_2} (\{s_0 ,s_2\})=-\{s_0\},   
         \\
    &
    \dfrac{\partial}{\partial  s_0} (\{s_0,s_1\})= \{s_1\}, 
   ~~~
    \dfrac{\partial}{\partial  s_0} (\{s_0,s_2\})= \{s_2\}, ~~~\dfrac{\partial}{\partial  s_0} (\{s_0,s_1,s_2\})=\{s_1,s_2\}, 
  \\
  & \dfrac{\partial}{\partial  s_1} (\{s_0,s_1 ,s_2\})=-\{s_0,s_2\}, 
   ~~~\dfrac{\partial}{\partial  s_2} (\{s_0,s_1, ,s_2\})=-\{s_0,s_1\}. 
    \end{eqnarray*}
    It  follows  that  $S(\mathcal{H}',\partial)=\{s_1,s_2\}$  and  consequently 
    \begin{eqnarray*}
    \mathcal{H}'\mid_{S(\mathcal{H}',\partial)}=  \{\{s_1\},  \{s_2\},  \{s_1,s_2\}\}.  
    \end{eqnarray*}
    Moreover,  
    \begin{eqnarray*}
  %&  d s_0 (\{s_0\}) = d s_0 (\{s_0,s_1\})= d s_0 (\{s_0,s_2\})= d s_0 (\{s_0,s_1,s_2\}) = 0,\\
%&    ds_1(\{s_1\})=ds_1    (\{s_0,s_1\})=ds_1 (\{s_0,s_1,s_2\})=0,\\
%& ds_2(\{s_2\})= ds_2(\{s_0,s_2\}) = ds_2 (\{s_0,s_1,s_2\})=0, \\
&ds_0 (\{s_1\})= \{s_0,s_1\},  ~~~  ds_0(\{s_2\}) =  \{s_0,s_2\}, \\
&  ds_1(\{s_0\})= -\{s_0,s_1\},  ~~~  ds_1(\{s_2\}) =  \{s_1,s_2\}, 
~~~  ds_1(\{s_0,s_2\})  = - \{s_0,s_1,s_2\},\\
&  ds_2(\{s_0\})=  -\{s_0,s_2\}, ~~~  ds_2(\{s_1\})= -\{s_1,s_2\}, 
~~~  ds_2 (\{s_0,s_1\})  =  \{s_0,s_1,s_2\}. 
    \end{eqnarray*}
     It  follows  that  $S(\mathcal{H}',d)=\{s_0\}$  and  consequently 
   $
    \mathcal{H}'\mid_{S(\mathcal{H}', d)}= \{\{s_0\}\}
    $.  
    
   \item
   Let  $\varphi_0:  S\longrightarrow  S$  be  
   $\varphi_0(s_0)=s_0$,  
   $\varphi_0(s_1)=s_2$  
   and 
   $\varphi_0(s_2)=s_1$.  
    Then  $\varphi_0$  induces  a    morphism   
    $\varphi:  \mathcal{H}\longrightarrow \mathcal{H}$  
    as  well as  a  morphism  
    $\varphi':  \mathcal{H}'\longrightarrow \mathcal{H}'$.  
     With  the  help  of   (1), 
      the  induced  simplicial  map  
     $\varphi\mid_{S(\mathcal{H},\partial)}$ 
     as  well as  
     the   induced  morphism  of  independence  hypergraphs 
     $\varphi\mid_{S(\mathcal{H},d)}$  
      is  the  identity  map  on  $\{s_0\}$.  
      With  the  help  of  (2),  
      the  induced   simplicial  map   
      $\varphi'\mid_{S(\mathcal{H}',\partial)}$ 
      is  given  by  
      $\varphi'\mid_{\{s_1,s_2\}} (\{s_1\})= \{s_2\}$, 
      $\varphi'\mid_{\{s_1,s_2\}} (\{s_2\})= \{s_1\}$ 
      and 
      $\varphi'\mid_{\{s_1,s_2\}} (\{s_1,s_2\})= \{s_1, s_2\}$. 
          The   induced  morphism  of  independence  hypergraphs 
     $\varphi'\mid_{S(\mathcal{H}',d)}$  
      is  the  identity  map  on  $\{s_0\}$.  
      
    \item
    Consider  the  canonical  inclusion  $\varphi:  \mathcal{H}\longrightarrow \mathcal{H}'$
    where $\varphi_0$  is  the  identity  map  on  $S$.  
     There  is  no  induced  simplicial  map  from  $\mathcal{H}\mid_{S(\mathcal{H},\partial)}$  
     to    $\mathcal{H}'\mid_{S(\mathcal{H}',\partial)}$.   
    \end{enumerate}
    \end{example}

    \section{The  constrained  homology }\label{s---7}

    Let  $S$  be  a  set  with  a  total  order $\prec$.  
    Let  $\mathcal{K}$  be  an  augmented  simplicial  complex on  $S$.  
    Restricting  (\ref{eq-wwwvvva1})  to  
    $R_*(\mathcal{K})$, 
    we  obtain   an  $R$-linear  map
    \begin{eqnarray*}
     \overrightarrow{\frac{\partial}{\partial  s}}:  && R_n(\mathcal{K})\longrightarrow  \prod_{n+1}R_{n-1}(\mathcal{K}). 
    \end{eqnarray*}  
        Let $\mathcal{L}$  be  an  augmented  independence  hypergraph  on  $S$. 
     Restricting  (\ref{eq-wwwvvva2})  to  
    $R_*(\mathcal{L})$, 
    we  obtain   an  $R$-linear  map
    \begin{eqnarray*}
    \overrightarrow{ds}:  &&   R_n(\mathcal{L})\longrightarrow  \prod_{n+2}   R_{n+1}(\mathcal{L}). 
    \end{eqnarray*}  
 Let  $k\in  \mathbb{N}$.  
 Let  $\alpha\in   \wedge^{2k+1} ( {\partial}/{\partial  s}\mid  s\in  S  )$  
 and    $\omega\in  \wedge^{2k+1}  (  d s \mid  s\in  S  )$.  
    Let  $q\in \mathbb{N}$.  
      We  have  a  sub-chain  complex 
   \begin{eqnarray*}
   \cdots   \overset{\alpha_{p+2}}{\longrightarrow}  R_{(p+1)(2k+1)+q}(\mathcal{K})  \overset{\alpha_{p+1}}{\longrightarrow} R_{p ( 2k+1)+q}(\mathcal{K})  \overset{\alpha_p}{\longrightarrow}  R_{(p-1)(2k+1)+q}(\mathcal{K})  \overset{\alpha_{p-1}}{\longrightarrow}   \cdots 
   \end{eqnarray*}
   of   
    $(\tilde{\mathcal{F}}\langle  S\rangle_*, \alpha,q)$,   denoted  as 
    $(\tilde  R_*(\mathcal{K}), \alpha,q)$,   
        and 
    a  sub-chain  complex 
      \begin{eqnarray*}
  \cdots   \overset{\omega_{p-2}}{\longrightarrow}  R_{(p-1)(2k+1)+q}(\mathcal{L})    \overset{\omega_{p-1}}{\longrightarrow}R_{p( 2k+1)+q}(\mathcal{L})  \overset{\omega_p}{\longrightarrow} R_{(p+1)(2k+1)+q}(\mathcal{L})  \overset{\omega_{p+1}}{\longrightarrow}   \cdots
   \end{eqnarray*}
   of   
   $(\tilde { \mathcal{F}}\langle  S\rangle_*, \omega,q)$,  denoted  as  $(\tilde  R_*(\mathcal{L}), \omega,q)$.  
      Note  that  
      \begin{eqnarray*}
      R_{-1}(\mathcal{K})= 
      \begin{cases}
      R  &{\rm~if~}  \emptyset\in\mathcal{K},\\
      0  &{\rm~if~}  \emptyset\notin\mathcal{K},
      \end{cases}
      ~~~~~~
            R_{-1}(\mathcal{L})= 
      \begin{cases}
      R  &{\rm~if~}  \emptyset\in\mathcal{L},\\
      0  &{\rm~if~}  \emptyset\notin\mathcal{L}.
      \end{cases}
      \end{eqnarray*}
            Consider  the  $R$-modules 
 \begin{eqnarray*}
  C((\tilde  R_*(\mathcal{K}), \alpha,q),   (\tilde  R_*(\mathcal{K}), \alpha,q-2l)), ~~~
   C((\tilde  R_*(\mathcal{L}), \omega,q),   (\tilde  R_*(\mathcal{L}), \omega,q+2l)).   
   \end{eqnarray*} 
 Take  the  direct  sums   
 \begin{eqnarray*}
  C_-(\tilde  R_*(\mathcal{K}), \alpha,q)&=&\bigoplus_{l=0}^\infty   C((\tilde  R_*(\mathcal{K}), \alpha,q),   (\tilde  R_*(\mathcal{K}), \alpha,q-2l)),\\
  C_+(\tilde  R_*(\mathcal{L}), \omega,q)&=&\bigoplus_{l=0}^\infty   C((\tilde  R_*(\mathcal{L}), \omega,q),   (\tilde  R_*(\mathcal{L}), \omega,q+2l)).
   \end{eqnarray*}
Define  the  multiplication   on  $  C_-(\tilde  R_*(\mathcal{K}), \alpha,q)$  
as  well  as  on  $C_+(\tilde  R_*(\mathcal{L}), \omega,q)$
  as  the  $R$-linear  extensions  of  the  composition
 of  chain  maps.

 \begin{lemma}
 \label{pr-bmaq17-thkl}
There   is  a   homomorphism  from  $\wedge^{2*}  ({\partial}/{\partial  s}\mid  s\in  S  )$   to   $ C_-(\tilde  R_*(\mathcal{K}), \alpha,q)$  and   a  homomorphism  from  $\wedge^{2*}   ( d s  \mid  s\in  S  )$   to   $ C_+(\tilde  R_*(\mathcal{L}), \omega,q)$.  % Consequently,     
 %there   is  a   homomorphism  from  $P  ({\partial}/{\partial  s}\mid  s\in  S(\mathcal{H},\partial)  )$   to   $ C_-(R_*(\mathcal{H}\mid_{S(\mathcal{H},\partial) }), \alpha,q)$  and   a  homomorphism  from  $P  ( d s  \mid  s\in  S  )$   to   $ C_+(R_*(\mathcal{L}), \omega,q)$. 
 \end{lemma}

\begin{proof}
 The proof  is  an  analogue  of  Theorem~\ref{pr-bmaq17}.   
\end{proof}

  The     constrained  homology  of  $\mathcal{K}$
   with  respect to  $\alpha$  and  $q$  is  (cf.  \cite[Definition~4.3]{camb2023})
   \begin{eqnarray*}
   H_n(\mathcal{K},\alpha,q)= {\rm  Ker}(\alpha_{n}) /  {\rm  Im } (\alpha_{n+1}),  ~~~  n\in \mathbb{N}
   \end{eqnarray*}
   and  the   constrained  cohomology  of  $\mathcal{L}$
   with  respect to  $\omega$  and  $q$  is  (cf.  \cite[Definition~4.4]{camb2023})
   \begin{eqnarray*}
   H^n(\mathcal{L},\omega,q)= {\rm  Ker}(\omega_{n}) /  {\rm  Im } (\omega_{n-1}),  ~~~  n\in \mathbb{N}.  
   \end{eqnarray*}
       Given  two  graded  $R$-modules  $G$  and  $H$,   let  ${\rm  Hom}(G,H)$  
 be  the  $R$-module  consisting  of  all  homomorphisms  from  $G$  to  $H$.      Consider  the  $R$-modules 
 \begin{eqnarray*}
  {\rm  Hom}((H_*(\mathcal{K}, \alpha,q),   H_*(\mathcal{K}, \alpha,q-2l)), ~~~
    {\rm  Hom}(H^*(\mathcal{L}, \omega,q),   H^*(\mathcal{L}, \omega,q+2l)).   
   \end{eqnarray*} 
 Take  the  direct  sums   
 \begin{eqnarray*}
  {\rm  Hom}_-(H_*(\mathcal{K}, \alpha,q))&=&\bigoplus_{l=0}^\infty    {\rm  Hom}((H_*(\mathcal{K}, \alpha,q),   H_*(\mathcal{K}, \alpha,q-2l)),\\
   {\rm  Hom}_+(H^*(\mathcal{L}, \omega,q))&=&\bigoplus_{l=0}^\infty   {\rm  Hom}(H^*(\mathcal{L}, \omega,q),   H^*(\mathcal{L}, \omega,q+2l)).
   \end{eqnarray*}
   Similar with  (\ref{eq-mmvdaq1}),   define  the  multiplication   $\varphi_1\circ\varphi_2$   for  any   
\begin{eqnarray*}
&\varphi_1\in  {\rm  Hom}((H_*(\mathcal{K}, \alpha,q-2l_2),   H_*(\mathcal{K}, \alpha,q-2(l_1+l_2))),\\
&\varphi_2\in  {\rm  Hom}((H_*(\mathcal{K}, \alpha,q),   H_*(\mathcal{K}, \alpha,q-2l_2)).  
\end{eqnarray*}  
Similar  with  (\ref{eq-mmvdaq2}),  define  the  multiplication  $\psi_1\circ\psi_2$   
for    any   
\begin{eqnarray*}
&\psi_1\in  {\rm  Hom}(H^*(\mathcal{L}, \omega,q+2l_2),   H^*(\mathcal{L}, \omega,q+2(l_1+l_2))),\\
&\psi_2\in  {\rm  Hom}(H^*(\mathcal{L}, \omega,q),   H^*(\mathcal{L}, \omega,q+2l_2)).  
\end{eqnarray*}

%Define  the  multiplication   on  $  {\rm  Hom}_-(H_*(\mathcal{K}, \alpha,q))$  
%as  well  as  on  $ {\rm  Hom}_+(H^*(\mathcal{L}, \omega,q))$
 % as  the  $R$-linear  extensions  of  the  composition
 %of  homomorphisms.   

    \begin{theorem}
   % [Main Result  I]
 \label{th-1daza}
There   is  a   homomorphism  from  $\wedge^{2*}  ({\partial}/{\partial  s}\mid  s\in  S  )$   to   ${\rm  Hom}_-(H_*(\mathcal{K}, \alpha,q))$  as  well as     a  homomorphism  from  $\wedge^{2*}   ( d s  \mid  s\in  S  )$   to   $  {\rm  Hom}_+(H^*(\mathcal{L}, \omega,q))$.  % Consequently,     
 %there   is  a   homomorphism  from  $P  ({\partial}/{\partial  s}\mid  s\in  S(\mathcal{H},\partial)  )$   to   $ C_-(R_*(\mathcal{H}\mid_{S(\mathcal{H},\partial) }), \alpha,q)$  and   a  homomorphism  from  $P  ( d s  \mid  s\in  S  )$   to   $ C_+(R_*(\mathcal{L}), \omega,q)$. 
 \end{theorem}
 
 \begin{proof}
A  chain  map  from  $(\tilde  R_*(\mathcal{K}), \alpha,q)$   to  
$  \tilde  R_*(\mathcal{K}), \alpha,q-2l)$   induces  a   homomorphism  
from  $ H_*(\mathcal{K}, \alpha,q)$  to  $   H_*(\mathcal{K}, \alpha,q-2l)$. 
The  composition   of  the  chain  maps  induces  the  composition  
of  the  homomorphisms  of  the  constrained homology  groups.  
 By  the  first  assertion  of  Lemma~\ref{pr-bmaq17-thkl},  there   is  a   homomorphism  from  $\wedge^{2*}  ({\partial}/{\partial  s}\mid  s\in  S  )$   to   ${\rm  Hom}_-(H_*(\mathcal{K}, \alpha,q))$.

 A  chain  map  from  $(\tilde  R_*(\mathcal{L}), \omega,q)$   to  
$  (\tilde  R_*(\mathcal{L}), \omega,q+2l)$   induces  a   homomorphism  
from  $ H^*(\mathcal{L}, \omega,q)$  to  $   H^*(\mathcal{L}, \omega,q+2l)$. 
The  composition   of  the  chain  maps  induces  the  composition   
of  the  homomorphisms  of  the  constrained cohomology  groups.  
 By  the  second   assertion  of  Lemma~\ref{pr-bmaq17-thkl},  there   is  a   homomorphism  from  $\wedge^{2*}  (d s \mid  s\in  S  )$   to   ${\rm  Hom}_+(H^*(\mathcal{L}, \omega,q))$.  
 \end{proof}
 
 \begin{theorem}%[Main Result  II]
\label{th-77.3}
 \begin{enumerate}[(1)]
 \item
 Let  $\mathcal{K}$  and  $\mathcal{K}'$  be  augmented  simplicial  complexes  on  $S$  such  that  
 $\mathcal{K}\subseteq  \mathcal{K}'$.  
 The  canonical  inclusion  $\iota: \mathcal{K}\longrightarrow  \mathcal{K}'$  
  induces  a  homomorphism  
  %\begin{eqnarray*}
  $\iota_*:  H_*(\mathcal{K}, \alpha, q)\longrightarrow  H_*(\mathcal{K}',\alpha,q)$  %
 % \end{eqnarray*}
  such  that  for  any  $\beta\in  \wedge^{2l}  ({\partial}/{\partial  s}\mid  s\in  S  ) $,  the  diagram  commutes
  \begin{eqnarray*}
  \xymatrix{
H_*(\mathcal{K}, \alpha,q) \ar[r]^-{\beta}  \ar[d]_-{\iota_*}&   H_*(\mathcal{K}, \alpha,q-2l)\ar[d]^-{\iota_*}\\
H_*(\mathcal{K}', \alpha,q)   \ar[r]^-{\beta}    &H_*(\mathcal{K}', \alpha,q-2l) ;  
  }
  \end{eqnarray*}
  \item
   Let  $\mathcal{L}$  and  $\mathcal{L}'$  be  augmented  independence  hypergraphs  on  $S$  such  that  
 $\mathcal{L}\subseteq  \mathcal{L}'$.  
 The  canonical  inclusion  $\iota: \mathcal{L}\longrightarrow  \mathcal{L}'$  
  induces  a  homomorphism  
  %\begin{eqnarray*}
  $\iota_*:  H^*(\mathcal{L}, \omega, q)\longrightarrow  H^*(\mathcal{L}',\omega,q)$  %
 % \end{eqnarray*}
  such  that  for  any  $\mu\in  \wedge^{2l}  (ds\mid  s\in  S  ) $,  the  diagram  commutes
  \begin{eqnarray*}
  \xymatrix{
H^*(\mathcal{L}, \omega,q) \ar[r]^-{\mu}  \ar[d]_-{\iota_*}&   H^*(\mathcal{L}, \omega,q+2l)\ar[d]^-{\iota_*}\\
H^*(\mathcal{L}',\omega,q)   \ar[r]^-{\mu}    &H^*(\mathcal{L}', \omega,q+2l).    
  }
  \end{eqnarray*}

  \end{enumerate}
 \end{theorem}
   \begin{proof}
 By  Corollary~\ref{co-008.3},   the  canonical  inclusion  $\iota:  \mathcal{K}\longrightarrow \mathcal{K}'$  induces 
   a  chain  map  $\iota_\#:  (\tilde  R_*(\mathcal{K}), \alpha,q)\longrightarrow  (\tilde  R_*(\mathcal{K}'), \alpha,q)  
   $
  such  that    the  diagram  commutes 
  \begin{eqnarray*}
  \xymatrix{
    (\tilde  R_*(\mathcal{K}), \alpha,q)\ar[r]^-{\beta} \ar[d]_-{\iota_\#}&  (\tilde  R_*(\mathcal{K}), \alpha,q-2l)\ar[d]^-{\iota_\#}\\
            (\tilde  R_*(\mathcal{K}'), \alpha,q)\ar[r]^-{\beta}&  (\tilde  R_*(\mathcal{K}), \alpha,q-2l).  
  }
  \end{eqnarray*} 
 Applying  the  homology  of  chain  complexes  to  the  above  diagram,  we  obtain  (1).

 By  Corollary~\ref{co-mzx},   the  canonical  inclusion  $\iota:  \mathcal{L}\longrightarrow \mathcal{L}'$  induces 
   a  chain  map  $\iota_\#:  (\tilde  R^*(\mathcal{L}), \omega,q)\longrightarrow  (\tilde  R^*(\mathcal{L}'), \omega,q)  
   $
  such  that    the  diagram  commutes 
  \begin{eqnarray*}
  \xymatrix{
    (\tilde  R_*(\mathcal{L}), \omega,q)\ar[r]^-{\mu} \ar[d]_-{\iota_\#}&  (\tilde  R_*(\mathcal{L}), \omega,q+2l)\ar[d]^-{\iota_\#}\\
            (\tilde  R_*(\mathcal{L}'), \omega,q)\ar[r]^-{\mu}&  (\tilde  R_*(\mathcal{L}), \omega,q+2l).  
  }
  \end{eqnarray*} 
 Applying  the  homology  of  chain  complexes  to  the  above  diagram,  we  obtain  (2). 
   \end{proof}
   
   Let  $\mathcal{K}_1$  and  $\mathcal{K}_2$  be  augmented  simplicial  complexes  on  $S$.  
We     have  
 a  long  exact  sequence  of  the  constrained  homology  groups
 \begin{eqnarray*}
\cdots \longrightarrow      H_n(\mathcal{K}_1 \cap  \mathcal{K}_2,\alpha, m)\longrightarrow 
  H_n (\mathcal{K}_1,\alpha, m) \oplus   H_n( \mathcal{K}_2,\alpha, m)\longrightarrow  \nonumber\\
   H_n(\mathcal{K}_1 \cup  \mathcal{K}_2,\alpha, m)\longrightarrow H_{n-1}(\mathcal{K}_1 \cap  \mathcal{K}_2,\alpha, m)\longrightarrow \cdots 
   \label{eq-lex1a}
 \end{eqnarray*} 
denoted  as   ${\bf  MV}_*( \mathcal{K}_1,\mathcal{K}_2, \alpha,m)$.   
   Let  $\mathcal{L}_1$  and  $\mathcal{L}_2$  be  augmented  independence  hypergraphs  on  $S$.  
We     have  
 a  long  exact  sequence  of  the  constrained  cohomology  groups
 \begin{eqnarray*}
\cdots \longrightarrow      H^n(\mathcal{L}_1 \cap  \mathcal{L}_2,\omega, m)\longrightarrow 
  H^n (\mathcal{L}_1,\omega, m) \oplus   H^n( \mathcal{L}_2,\omega, m)\longrightarrow  \nonumber \\
   H^n(\mathcal{L}_1 \cup  \mathcal{L}_2,\omega, m)\longrightarrow H^{n+1}(\mathcal{L}_1 \cap  \mathcal{L}_2,\omega, m)\longrightarrow \cdots 
   \label{eq-lex2b}
 \end{eqnarray*} 
denoted  as   ${\bf  MV}^*(\mathcal{L}_1,\mathcal{L}_2, \omega,m)$.   
For  any   augmented  simplicial  complexes  $\mathcal{K}_1$  and  $\mathcal{K}_2$     on  $S$,   
   we  have  a   long  exact  sequence   ${\bf  MV}_*(\mathcal{K}_1,\mathcal{K}_2, \alpha,m)$ 
   as  well  as  long  exact  sequences  ${\bf  MV}^*(\Gamma_S\mathcal{K}_1,\Gamma_S\mathcal{K}_2, \omega,m)$  and  ${\bf  MV}^*(\gamma_S\mathcal{K}_1,\gamma_S\mathcal{K}_2, \omega,m)$.  
    For  any   augmented  simplicial  complexes  $\mathcal{K}_1\subseteq\mathcal{K}'_1$  and  $\mathcal{K}_2\subseteq\mathcal{K}'_2$     on  $S$,  
   the  canonical  inclusions  $\iota$  induce  morphisms  of   long  exact  sequences  
\begin{eqnarray}
\iota:&& {\bf  MV}_*(\mathcal{K}_1,\mathcal{K}_2, \alpha,m)\longrightarrow {\bf  MV}_*(\mathcal{K}'_1,\mathcal{K}'_2, \alpha,m), \nonumber\\
 && {\bf  MV}^*(\Gamma_S\mathcal{K}'_1, \Gamma_S\mathcal{K}'_2, \omega,m)\longrightarrow {\bf  MV}^*(\Gamma_S\mathcal{K}_1,\Gamma_S\mathcal{K}_2, \omega,m), \nonumber\\
 && {\bf  MV}^*(\gamma_S\mathcal{K}'_1, \gamma_S\mathcal{K}'_2, \omega,m)\longrightarrow {\bf  MV}^*(\gamma_S\mathcal{K}_1,\gamma_S\mathcal{K}_2, \omega,m).    
 \label{eq-7.zzzmmm1}
\end{eqnarray}   
  For  any   augmented  independence  hypergraphs  $\mathcal{L}_1$  and  $\mathcal{L}_2$     on  $S$,  
   we  have  a   long  exact  sequence   ${\bf  MV}^*(\mathcal{L}_1,\mathcal{L}_2, \omega,m)$ 
   as  well  as  long  exact  sequences  ${\bf  MV}_*(\Gamma_S\mathcal{L}_1,\Gamma_S\mathcal{L}_2, \alpha,m)$  and  ${\bf  MV}_*(\gamma_S\mathcal{L}_1,\gamma_S\mathcal{L}_2, \alpha,m)$.  
 For  any   augmented  independence  hypergraphs  $\mathcal{L}_1\subseteq\mathcal{L}'_1$  and  $\mathcal{L}_2\subseteq\mathcal{L}'_2$     on  $S$,  
   the  canonical  inclusions  $\iota$  induce  morphisms  of   long  exact  sequences  
\begin{eqnarray}
\iota:&& {\bf  MV}^*(\mathcal{L}_1,\mathcal{L}_2, \omega,m)\longrightarrow {\bf  MV}^*(\mathcal{L}'_1,\mathcal{L}'_2, \omega,m), \nonumber\\
 && {\bf  MV}_*(\Gamma_S\mathcal{L}'_1, \Gamma_S\mathcal{L}'_2, \alpha,m)\longrightarrow {\bf  MV}_*(\Gamma_S\mathcal{L}_1,\Gamma_S\mathcal{L}_2, \alpha,m),\nonumber\\
 &&   {\bf  MV}_*(\gamma_S\mathcal{L}'_1, \gamma_S\mathcal{L}'_2, \alpha,m)\longrightarrow {\bf  MV}_*(\gamma_S\mathcal{L}_1,\gamma_S\mathcal{L}_2, \alpha,m).   
  \label{eq-7.zzzmmm2}
\end{eqnarray}

\begin{example}
Let  $S=\{s_0,s_1, s_2\}$.  
Let  $\alpha =  \sum_{i=0}^3   r_i   \partial/\partial  s_i$   such  that  $r_0,r_1,r_2 \neq  0$.  

\begin{enumerate}[(1)]
\item
Let  $\mathcal{K} = \{ \{s_0\},  \{s_1\},    \{s_0,s_1\}\}$.  
Then  
\begin{eqnarray*}
& \dfrac{\partial }{\partial  s_0}(\{s_0\}) =  \dfrac{\partial }{\partial  s_0}(\{s_1\})= \dfrac{\partial }{\partial  s_1}(\{s_0\}) =\dfrac{\partial }{\partial  s_1}(\{s_1\}) =0,  \\
& \dfrac{\partial }{\partial  s_0}( \{s_0,s_1\})= \{s_1\}, ~~~~~~ 
\dfrac{\partial }{\partial  s_1}( \{s_0,s_1\})=- \{s_0\}
\end{eqnarray*}
which  implies 
\begin{eqnarray*}
{\rm  Im}(\alpha:  R_1(\mathcal{K})\longrightarrow  R_0(\mathcal{K}) )  &=& R ( r_0  \{s_1\} -  r_1\{s_0\}), \\
{\rm  Ker}(\alpha:  R_0(\mathcal{K})\longrightarrow  0 )  &=& R  (  \{s_0\},   \{s_1\}),\\
{\rm  Ker}(\alpha:  R_1(\mathcal{K})\longrightarrow  R_0(\mathcal{K}) )  &=&0
\end{eqnarray*}
and  consequently 
\begin{eqnarray*}
H_n(\mathcal{K}, \alpha, q)= 
\begin{cases}
R,   
&  n=0{\rm~and~} q=0;\\
0,  & {\rm~otherwise}.  
\end{cases} 
\end{eqnarray*}
\item
Let  $\mathcal{K}' = \{\emptyset, \{s_0\},  \{s_1\},  \{s_2\},  \{s_0,s_1\},  \{s_1,s_2\}, \{s_0,s_2\}\}$.  
Then  
\begin{eqnarray*}
&\dfrac{\partial}{\partial  s_i}( \{s_i\})=    \emptyset, ~~~~~~ i=0,1,2, \\
&\dfrac{\partial}{\partial  s_i}( \{s_j\})=0,   ~~~~~~ 0  \leq   i\neq  j  \leq   2,\\
&\dfrac{\partial}{\partial  s_i}( \{s_j,s_k\})=  \begin{cases}
0,  &i\neq  j, k;\\
\{s_k\},  &  i=j;\\
-\{s_j\},  &  i=k
\end{cases}
~~~~~~  0\leq  i,j,k\leq  2{\rm~and~}  j<k
\end{eqnarray*}
which  implies  
\begin{eqnarray*}
{\rm  Im}(\alpha:  R_0(\mathcal{K}')\longrightarrow  R_{-1}(\mathcal{K}') )  &=&   R ((r_0+r_1+r_2)\emptyset),\\
{\rm  Im}(\alpha:  R_1(\mathcal{K}')\longrightarrow  R_0(\mathcal{K}') )  &=& R  ( r_0  \{s_1\} -  r_1\{s_0\},  r_1\{s_2\}- r_2\{s_1\},\\
&& r_0\{s_2\} - r_2\{s_0\} ), \\
{\rm  Ker}(\alpha:  R_0(\mathcal{K}')\longrightarrow  R_{-1}(\mathcal{K}') )  &=&\{   x_0\{s_0\}+
  x_1   \{s_1\}+  x_2 \{s_2\}\mid \\
  &&    x_0,x_1,x_2\in   R {\rm~such~that~}    x_0+x_1+x_2=0 \},\\
{\rm  Ker}(\alpha:  R_1(\mathcal{K}')\longrightarrow  R_0(\mathcal{K}') )  &=&R( 
r_0\{s_1,s_2\} 
  +r_1  \{s_0,s_2\}+   r_2\{s_0,s_1\} )
\end{eqnarray*}
and  consequently  
\begin{eqnarray*}
H_n(\mathcal{K}', \alpha, q)= 
\begin{cases}
R/ (r_0+r_1+r_2)R,  &n=-1{\rm~and~}q=0;\\
R,   
&  n+q=1;\\
0,  & {\rm~otherwise}.  
\end{cases} 
\end{eqnarray*}
\item
Let  $\iota:  \mathcal{K}\longrightarrow \mathcal{K}'$  be  the  canonical  inclusion   induced  by  the  identity  map   on  $S$,  i.e.  $\iota_0={\rm  id}_S$.  
 Then  the   induced  homomorphism   $\iota_*:  H_*(\mathcal{K},\alpha,q)\longrightarrow  H_*(\mathcal{K}',\alpha,q)$   
(cf.  Theorem~\ref{th-77.3}~(1))  
 is  the  zero  map.  
\end{enumerate}

\end{example}

\begin{example}
Let  $S=\{s_0,s_1, s_2\}$.  
Let   
$\omega  = \sum_{i=0}^2   r_i  ds_i $  such  that  $r_0,r_1,r_2\neq  0$.
\begin{enumerate}[(1)]
\item
Let  $\mathcal{L}=\{\{s_0,s_1\},  \{s_0,s_1,s_2\}\}$.   Then  
\begin{eqnarray*}
&  ds_0 (\{s_0,s_1\}) =  ds_1 (\{s_0,s_1\} )=ds_i(\{s_0,s_1,s_2\})=  0, \\
&  ds_2(\{s_0,s_1\})= \{s_0,s_1,s_2\},
\end{eqnarray*}
$i=0,1,2$,   
which  implies 
\begin{eqnarray*}
{\rm  Im}(\omega:  R_1(\mathcal{L})\longrightarrow  R_2(\mathcal{L}) )=  R(r_2 \{s_0,s_1,s_2\})
\end{eqnarray*}
and  consequently  
\begin{eqnarray*}
H^n(\mathcal{L},\omega,q)=\begin{cases}
R/  r_2 R,  & n+q = 2;\\
0, &{\rm~otherwise}. 
\end{cases}
\end{eqnarray*}
\item
Let  $\mathcal{L}'=\{\{s_0,s_1\}, \{s_0,s_2\},   \{s_0,s_1,s_2\}\}$.  
Then  
\begin{eqnarray*}
&  ds_0 (\{s_0,s_1\})=  ds_1 (\{s_0,s_1\} ) \\
&=ds_0 (\{s_0,s_2\})= ds_2 (\{s_0,s_2\})  =ds_i(\{s_0,s_1,s_2\})=  0,\\
&ds_2(\{s_0,s_1\})=-ds_1(\{s_0,s_2\})= \{s_0,s_1,s_2\}
\end{eqnarray*}
which  implies  
\begin{eqnarray*}
&{\rm  Im}(\omega:  R_1(\mathcal{L}')\longrightarrow  R_2(\mathcal{L}') )=  (r_1,r_2) (\{s_0,s_1,s_2\}),\\
&{\rm   Ker}  (\omega:  R_1(\mathcal{L}')\longrightarrow  R_2(\mathcal{L}') ) =
R((r_1-r_2)  (\{s_0,s_1,s_2\})), 
\end{eqnarray*}
where  $(r_1,r_2)$  is  the  ideal  generated  by  $\{r_1,r_2\}$,   and   consequently  
\begin{eqnarray*}
H^n(\mathcal{L}',\omega,q)=\begin{cases}
R/  (r_1, r_2),  & n+q = 2;\\
0, &{\rm~otherwise}. 
\end{cases}
\end{eqnarray*}  

\item
Let  $\iota:  \mathcal{L}\longrightarrow  \mathcal{L}'$  be  the  canonical  inclusion  induced  by  the  identity  map   on  $S$,  i.e.  $\iota_0={\rm  id}_S$.   
Then  the   induced  homomorphism   $\iota_*:  H^n(\mathcal{L},\omega,q)\longrightarrow  H^n(\mathcal{L}',\omega,q)$   
(cf.  Theorem~\ref{th-77.3}~(2))  
 is  the     canonical  inclusion  $R/r_2 R\longrightarrow  R/ (r_1,r_2)$  for  $n+q=2$.  
\end{enumerate}
\end{example}

\section{The   constrained  persistent   homology}\label{s--8}

Let  $\{\mathcal{K}_x\}_{x\in\mathbb{R}}$  be  a  filtration  of  augmented  simplicial  complexes  such  that  
    \begin{enumerate}[(1).]
    \item
    for  any   $x\in  \mathbb{R}$,   $\mathcal{K}_x$  is  an  augmented    simplicial complex  with its  vertices from  $S$; 
    \item  
   for  any  $-\infty<x\leq  y<+\infty$,  there  is  a  canonical  inclusion  $\iota_x^y:   \mathcal{K}_x\longrightarrow \mathcal{K}_y$    induced  by  the  identity  map  on  $S$    satisfying  $\iota_x^x= {\rm  id}$   for  any  $x\in\mathbb{R}$  and    $\iota_x^z  =  \iota_y^z \circ  \iota_x^y$  for  any   $-\infty<x\leq  y\leq  z<+\infty$.   
      \end{enumerate}
 Let $\alpha\in {\rm Ext}_{2t+1}(S)$  and  $\beta\in {\rm Ext}_{2s}(S)$.          We  have  a  family  of  constrained  homology  groups 
 \begin{eqnarray}\label{eq-9.876}
 \{H_n(\mathcal{K}_x,\alpha,m)\}_{x\in\mathbb{R}}
 \end{eqnarray}
   and  a  family  of  homomorphisms  of  homology  groups 
 \begin{eqnarray}\label{eq-9.878}
 \{(\iota_x^y)_*:  H_n(\mathcal{K}_x,\alpha,m)\longrightarrow  H_n(\mathcal{K}_y,  \alpha, m)\}_{-\infty<x\leq  y<+\infty}.   
      \end{eqnarray}
      For  any  $-\infty<x\leq  y<+\infty$,   by  Theorem~\ref{th-77.3}~(1),   we  have  a  commutative diagram 
 \begin{eqnarray}\label{diag-zzm1}
 \xymatrix{
 H_n(\mathcal{K}_x,\alpha,m) \ar[r]  ^-{\beta_*}  \ar[d]_{(\iota_x^y)_*}& H_n(\mathcal{K}_x,\alpha,m-2s) \ar[d]^{(\iota_x^y)_*}\\
 H_n(\mathcal{K}_y, \alpha,m) \ar[r]  ^-{\beta_*}  & H_n(\mathcal{K}_y,  \alpha,m-2s).    
 }
 \end{eqnarray}
    We  call  (\ref{eq-9.876})  together  with  (\ref{eq-9.878})  the $n$-th  {\it  constrained  persistent homology}   of   
      $\{\mathcal{K}_x\}_{x\in\mathbb{R}}$  with  respect   to   $\alpha$  and  $m$  and  denote  it  as   
 ${\bf  H}_n(\mathcal{K}_x,\alpha,m\mid  x\in\mathbb{R})$.         
 By  the  commutative  diagram  (\ref{diag-zzm1}),  
   we  have  an  induced   homomorphism  of  persistent  $R$-modules    (cf.  \cite[Section~1.3,  Module categories]{pmd})
      \begin{eqnarray*}
      \beta_*:  ~~~{\bf  H}_n(\mathcal{K}_x,\alpha,m\mid  x\in\mathbb{R})\longrightarrow  {\bf  H}_n(\mathcal{K}_x,\alpha,m-2s\mid  x\in\mathbb{R}).  
      \end{eqnarray*}
          Let  $\{\mathcal{L}_x\}_{x\in\mathbb{R}}$  be  a  filtration  of   augmented    independence  hypergraphs  such  that  
    \begin{enumerate}[(1).]
    \item
    for  any   $x\in  \mathbb{R}$,   $\mathcal{L}_x$  is  an   augmented   independence  hypergraph  with its  vertices from  $S$; 
    \item  
   for  any  $-\infty<x\leq  y<+\infty$,  there  is  a  canonical  inclusion  $\theta_x^y:   \mathcal{L}_x\longrightarrow \mathcal{L}_y$  induced  by  the  identity  map  on  $S$   satisfying  $\theta_x^x= {\rm  id}$   for  any  $x\in\mathbb{R}$  and    $\theta_x^z  =  \theta_y^z \circ \theta_x^y$  for  any   $-\infty<x\leq  y\leq  z<+\infty$.      
      \end{enumerate}  
 Let $\omega\in {\rm Ext}^{2t+1}(S)$  and  $\mu\in {\rm Ext}^{2s}(S)$.     We  have  a  family  of  constrained   cohomology  groups 
 \begin{eqnarray}\label{eq-9.976}
 \{H^n(\mathcal{L}_x,\omega,m)\}_{x\in\mathbb{R}}
 \end{eqnarray}
   and  a  family  of  homomorphisms  of  cohomology   groups  
 \begin{eqnarray}\label{eq-9.978}
 \{(\theta_x^y)_*:  H^n(\mathcal{L}_x,\omega,m)\longrightarrow  H^n(\mathcal{L}_y,  \omega, m)\}_{-\infty<x\leq  y<+\infty}.   
      \end{eqnarray}
       For  any  $-\infty<x\leq  y<+\infty$,   by  Theorem~\ref{th-77.3}~(2),  we  have  a  commutative diagram 
 \begin{eqnarray}\label{diag-zzm2}
 \xymatrix{
 H^n(\mathcal{L}_x,\omega,m) \ar[r]  ^-{\mu_*}  \ar[d]_{(\theta_x^y)_*}& H^n(\mathcal{L}_x,\omega,m+2s) \ar[d]^{(\theta_x^y)_*}\\
 H^n(\mathcal{L}_y, \omega),m  \ar[r]  ^-{\mu_*}  & H^n(\mathcal{L}_y,\omega,m+2s).    
 }
 \end{eqnarray}
      We  call  (\ref{eq-9.976})  together  with  (\ref{eq-9.978})  the $n$-th  {\it  constrained  persistent  cohomology}   of   
      $\{\mathcal{L}_x\}_{x\in\mathbb{R}}$  with  respect   to   $\omega$  and  $m$  and  denote  it  as   
 ${\bf  H}^n(\mathcal{L}_x,\omega,m\mid  x\in\mathbb{R})$.       
 By  the  commutative  diagram  (\ref{diag-zzm2}),  
   we  have  an  induced   homomorphism  of  persistent  $R$-modules  
      \begin{eqnarray*}
      \mu_*: ~~~ {\bf  H}^n(\mathcal{L}_x,\omega,m\mid  x\in\mathbb{R})\longrightarrow  {\bf  H}^n(\mathcal{L}_x,\omega,m+2s\mid  x\in\mathbb{R}).  
      \end{eqnarray*}

       Given  two  graded  persistent   $R$-modules  ${\bf  G}$  and  ${\bf  H}$,   let  ${\rm  Hom}({\bf G},{\bf  H})$  
 be  the  $R$-module  consisting  of  all  persistent  homomorphisms  from  ${\bf  G}$  to  ${\bf  H}$.      Consider  the  $R$-modules 
 \begin{eqnarray*}
 && {\rm  Hom}( {\bf  H}_*(\mathcal{K}_x, \alpha,q\mid  x\in \mathbb{R}),   {\bf  H}_*(\mathcal{K}_x, \alpha,q-2l\mid  x\in \mathbb{R})), \\
 &&   {\rm  Hom}({\bf  H}^*(\mathcal{L}_x, \omega,q\mid  x\in \mathbb{R}),  {\bf  H}^*(\mathcal{L}_x, \omega,q+2l\mid  x\in \mathbb{R})).   
   \end{eqnarray*} 
 Take  the  direct  sums   
 \begin{eqnarray*}
  {\rm  Hom}_-({\bf  H}_*(\mathcal{K}_x, \alpha,q\mid  x\in \mathbb{R}))&=&\bigoplus_{l=0}^\infty     {\rm  Hom}( {\bf  H}_*(\mathcal{K}_x, \alpha,q\mid  x\in \mathbb{R}),   {\bf  H}_*(\mathcal{K}_x, \alpha,q-2l\mid  x\in \mathbb{R})),\\
   {\rm  Hom}_+({\bf  H}^*(\mathcal{L}_x, \omega,q\mid  x\in \mathbb{R}))&=&\bigoplus_{l=0}^\infty   {\rm  Hom}({\bf  H}^*(\mathcal{L}_x, \omega,q\mid  x\in \mathbb{R}),  {\bf  H}^*(\mathcal{L}_x, \omega,q+2l\mid  x\in \mathbb{R})).
   \end{eqnarray*}
   Similar with  (\ref{eq-mmvdaq1}),   define  the  multiplication   $\varphi_1\circ\varphi_2$   for  any   
\begin{eqnarray*}
&\varphi_1\in  {\rm  Hom}( {\bf  H}_*(\mathcal{K}_x, \alpha,q-2l_2\mid  x\in \mathbb{R}),   {\bf  H}_*(\mathcal{K}_x, \alpha,q-2(l_1+l_2)\mid  x\in \mathbb{R})),\\
&\varphi_2\in {\rm  Hom}( {\bf  H}_*(\mathcal{K}_x, \alpha,q\mid  x\in \mathbb{R}),   {\bf  H}_*(\mathcal{K}_x, \alpha,q-2l_2\mid  x\in \mathbb{R})).  
\end{eqnarray*}  
Similar  with  (\ref{eq-mmvdaq2}),  define  the  multiplication  $\psi_1\circ\psi_2$   
for    any   
\begin{eqnarray*}
&\psi_1\in  {\rm  Hom}({\bf  H}^*(\mathcal{L}_x, \omega,q+2l_2\mid  x\in \mathbb{R}),  {\bf  H}^*(\mathcal{L}_x, \omega,q+2(l_1+l_2)\mid  x\in \mathbb{R})),\\
&\psi_2\in  {\rm  Hom}({\bf  H}^*(\mathcal{L}_x, \omega,q\mid  x\in \mathbb{R}),  {\bf  H}^*(\mathcal{L}_x, \omega,q+2l_2\mid  x\in \mathbb{R})).  
\end{eqnarray*}
%Define  the  multiplication   on  $    {\rm  Hom}_-({\bf  H}_*(\mathcal{K}_x, \alpha,q\mid  x\in \mathbb{R}))$  
%as  well  as  on  $ {\rm  Hom}_+({\bf  H}^*(\mathcal{L}_x, \omega,q\mid  x\in \mathbb{R}))$
 % as  the  $R$-linear  extensions  of  the  composition
% of  homomorphisms.  
  The  next  corollary  follows  from  Theorem~\ref{th-1daza}   immediately.  
 \begin{corollary}\label{co-8.1}
There   is  a   homomorphism  from  $\wedge^{2*}  ({\partial}/{\partial  s}\mid  s\in  S  )$   to   ${\rm  Hom}_-({\bf  H}_*(\mathcal{K}_x, \alpha,q\mid  x\in \mathbb{R}))$  as  well  as    a  homomorphism  from  $\wedge^{2*}   ( d s  \mid  s\in  S  )$   to   $  {\rm  Hom}_+({\bf  H}^*(\mathcal{L}_x, \omega,q\mid  x\in \mathbb{R}))$.
\end{corollary}

     Let   
          $\{\mathcal{K}_x\}_{x\in\mathbb{R}}$  and  $\{\mathcal{K}'_x\}_{x\in\mathbb{R}}$  
          be  two  filtrations  of   augmented    simplicial complexes 
           where  the  canonical  inclusions  are   induced  by  the  identity  map 
           on  the  vertex  set   $S$.    
          Then  both  $\{\mathcal{K}_x\cap \mathcal{K}'_x\}_{x\in\mathbb{R}}$  and  
            $\{\mathcal{K}_x\cup \mathcal{K}'_x\}_{x\in\mathbb{R}}$   
            are  filtrations  of   augmented    simplicial  complexes  where  the  canonical  inclusions  are
             induced  by  the  identity  map 
           on   $S$.   
            For  any  $-\infty<x\leq  y<+\infty$,  the  canonical  inclusions  
            $\iota_x^y: \mathcal{K}_x\longrightarrow  \mathcal{K}_y$  and 
              ${\iota'}_x^y: \mathcal{K}'_x\longrightarrow  \mathcal{K}'_y$  
              induces  a  morphism  of  the Mayer-Vietoris 
            sequences   of the constrained  homology  
            \begin{eqnarray}\label{eq-1oiu}
            (\iota_x^y, {\iota'}_x^y)_*: ~~~  {\bf  MV}_*(\mathcal{K}_x,\mathcal{K}'_x,  \alpha,m)\longrightarrow {\bf  MV}_*(\mathcal{K}_y,\mathcal{K}'_y,\alpha,m). 
            \end{eqnarray}
            We  call the  family  $\{{\bf  MV}_*(\mathcal{K}_x,\mathcal{K}'_x,\alpha,m)\}_{x\in\mathbb{R}}$   together with the  family  of  morphisms     given by  (\ref{eq-1oiu})  the  
            {\it    persistent  Mayer-Vietoris  sequence}  of  the  constrained  persistent homology  for the filtrations $\{\mathcal{K}_x\}_{x\in\mathbb{R}}$ and  $\{\mathcal{K}'_x\}_{x\in\mathbb{R}}$  and  denote it  as  
            \begin{eqnarray*}
            {\bf   PMV}_*(\mathcal{K}_x,\mathcal{K}'_x,  \alpha,m\mid  x\in\mathbb{R}).  
            \end{eqnarray*}    
           Similarly,  let   
          $\{\mathcal{L}_x\}_{x\in\mathbb{R}}$  and  $\{\mathcal{L}'_x\}_{x\in\mathbb{R}}$  
          be  two  filtrations  of   augmented    independence hypergraphs    where  the  canonical  inclusions  are   induced  by  the  identity  map 
           on  the  vertex  set   $S$.    
           Then  both  $\{\mathcal{L}_x\cap \mathcal{L}'_x\}_{x\in\mathbb{R}}$  and  
            $\{\mathcal{L}_x\cup \mathcal{L}'_x\}_{x\in\mathbb{R}}$   
            are  filtrations  of   augmented    independence  hypergraphs     where  the  canonical  inclusions  are   induced  by  the  identity  map 
           on  the  vertex  set   $S$.    
            For  any  $-\infty<x\leq  y<+\infty$,  the  canonical  inclusions  
            $\theta_x^y: \mathcal{L}_x\longrightarrow  \mathcal{L}_y$  and 
              ${\theta'}_x^y: \mathcal{L}'_x\longrightarrow  \mathcal{L}'_y$  
              induces  a  morphism  of  the Mayer-Vietoris 
            sequences   of the constrained  homology  
            \begin{eqnarray}\label{eq-2oiu}
            (\theta_x^y,{\theta'}_x^y)_*: ~~~  {\bf  MV}^*(\mathcal{L}_x,\mathcal{L}'_x,  \omega,m)\longrightarrow {\bf  MV}^*(\mathcal{L}_y,\mathcal{L}'_y,\omega,m). 
            \end{eqnarray}
            We  call the  family  $\{{\bf  MV}^*(\mathcal{L}_x,\mathcal{L}'_x,\alpha,m)\}_{x\in\mathbb{R}}$   together with the  family  of  morphisms     given by  (\ref{eq-2oiu})  the  
            {\it    persistent  Mayer-Vietoris  sequence}  of  the  constrained  persistent  cohomology  for the filtrations $\{\mathcal{L}_x\}_{x\in\mathbb{R}}$ and  $\{\mathcal{L}'_x\}_{x\in\mathbb{R}}$  and  denote it  as  
            \begin{eqnarray*}
            {\bf   PMV}^*(\mathcal{L}_x,\mathcal{L}'_x,  \omega,m\mid  x\in\mathbb{R}).  
            \end{eqnarray*}

              The  next  corollary  follows  from   (\ref{eq-7.zzzmmm1})    immedaitely.
            \begin{corollary}\label{co-8.2}
              For  any  filtrations    $\{\mathcal{K}_x\}_{x\in\mathbb{R}}$  and  $\{\mathcal{K}'_x\}_{x\in\mathbb{R}}$  
          of   augmented    simplicial complexes 
           where  the  canonical  inclusions  are   induced  by  the  identity  map 
           on  the  vertices,   
              we  have  persistent  Mayer-Vietoris  sequences 
               \begin{eqnarray*}
                   &        {\bf   PMV}_*(\mathcal{K}_x,\mathcal{K}'_x,  \alpha,m\mid  x\in\mathbb{R}),\\
 & {\bf  PMV}^*(\Gamma_S\mathcal{K}_x, \Gamma_S\mathcal{K}'_x, \omega,m\mid  x\in \mathbb{R}), \\
 & {\bf  PMV}^*(\gamma_S\mathcal{K}_x, \gamma_S\mathcal{K}'_x, \omega,m\mid  x\in\mathbb{R}).  
\end{eqnarray*}   
            \end{corollary}        
            
            The  next  corollary  follows  from   (\ref{eq-7.zzzmmm2})  immedaitely.  
            \begin{corollary}\label{co-8.3}
              For  any  filtrations    $\{\mathcal{L}_x\}_{x\in\mathbb{R}}$  and  $\{\mathcal{L}'_x\}_{x\in\mathbb{R}}$  
          of   augmented    independence  hypergraphs  
           where  the  canonical  inclusions  are   induced  by  the  identity  map 
           on  the  vertices,   
              we  have  persistent  Mayer-Vietoris  sequences 
               \begin{eqnarray*}
                   &       {\bf   PMV}^*(\mathcal{L}_x,\mathcal{L}'_x,  \omega,m\mid  x\in\mathbb{R}),\\
 & {\bf  PMV}_*(\Gamma_S\mathcal{L}_x, \Gamma_S\mathcal{L}'_x, \alpha,m\mid  x\in \mathbb{R}), \\
 & {\bf  PMV}_*(\gamma_S\mathcal{K}_x, \gamma_S\mathcal{K}'_x, \alpha,m\mid  x\in\mathbb{R}).  
\end{eqnarray*}   
            \end{corollary}

 \bigskip
 
Shiquan Ren %(for correspondence)

Address:   School of Mathematics and Statistics,  Henan University,  Kaifeng  475004,  China. 

E-mail:  renshiquan@henu.edu.cn

\end{document}